\documentclass[a4paper]{amsart}

\usepackage{color,graphicx,enumerate,wrapfig,amssymb,mathtools}
\usepackage{epsfig}

\usepackage{eucal}

\newcommand{\R}{{\mathbb R}}

\newcommand{\N}{{\mathbb N}}

\newcommand{\Hf}{{\mathbb{H}}}

\newcommand{\cH}{{\mathcal H}}

\newcommand{\uu}{\mathbf{u}}
\newcommand{\vv}{\mathbf{v}}
\newcommand{\w}{\mathbf{w}}
\newcommand{\cc}{\mathbf{c}}
\newcommand{\h}{\mathbf{h}}
\newcommand{\g}{\mathbf{ g}}
\newcommand{\q}{\mathbf{q}}
\newcommand{\pp}{\mathbf{p}}
\newcommand{\wuuj}{\widetilde{\mathbf{u}_j}}
\newcommand{\wuuz}{\widetilde{\mathbf{u}_0}}

\newcommand{\e}{\varepsilon}
\newcommand{\al}{\alpha}
\newcommand{\be}{\beta}
\newcommand{\de}{\delta}
\newcommand{\si}{\sigma}
\newcommand{\vp}{\varphi}

\newcommand{\dist}{\operatorname{dist}}

\newcommand{\loc}{\operatorname{loc}}
\newcommand{\D}{\nabla}
\newcommand{\p}{\partial}

\newcommand{\mean}[1]{\langle{#1}\rangle}
\DeclareRobustCommand{\rchi}{{\mathpalette\irchi\relax}}
\newcommand{\irchi}[2]{\raisebox{\depth}{$#1\chi$}}

\usepackage{amsmath}
\usepackage{amsfonts}
\usepackage{caption}
\usepackage{subcaption}
\usepackage{esint}

\newtheorem{theorem}{Theorem}
\theoremstyle{plain}
\newtheorem{corollary}{Corollary}
\newtheorem{definition}{Definition}
\newtheorem{example}{Example}
\newtheorem{lemma}{Lemma}
\newtheorem{remark}{Remark}
\newtheorem{proposition}{Proposition}

\numberwithin{equation}{section}

\begin{document}

\author{Daniela De Silva}
\address[Daniela De Silva]
{Department of Mathematics \newline 
\indent  Barnard College, Columbia University \newline 
\indent New York, NY,
10027} 
\email[Daniela De Silva]{desilva@math.columbia.edu}

\author{Seongmin Jeon}
\address[Seongmin Jeon]
{Department of Mathematics \newline 
\indent KTH Royal Institute of Technology \newline 
\indent 100 44 Stockholm, Sweden} 
\email[Seongmin Jeon]{seongmin@kth.se}

\author{Henrik Shahgholian}
\address[Henrik Shahgholian]
{Department of Mathematics \newline 
\indent KTH Royal Institute of Technology \newline 
\indent 100 44 Stockholm, Sweden} 
\email[Henrik Shahgholian]{henriksh@kth.se}

\title[Almost minimizers for a singular system with free boundary]
{Almost minimizers for a singular system with free boundary}

\date{\today}
%35R35  	Free boundary problems
%35J60  	Nonlinear elliptic equations
\thanks{ DD is supported by a NSF grant (RTG 1937254).
 HS is supported by Swedish Research Council.}

\begin{abstract}
    In this paper we study  vector-valued almost minimizers of the energy functional
    $$
    \int_D\left(|\D\uu|^2+2|\uu|\right)\,dx .
    $$
    We establish the regularity for both minimizers  and   the "regular" part of the free boundary. The analysis of the free boundary is based on Weiss-type monotonicity formula and the epiperimetric inequality for the energy minimizers.
\end{abstract}

\maketitle 

\tableofcontents

\section{Introduction}
\subsection{Singular cooperative system}
Let $D\subset\R^n$, $n\ge 2$, be an open set and consider the problem of minimizing the energy \begin{equation}\label{functional}
\int_D\left(|\D\uu|^2+2|\uu|\right)\,dx
\end{equation}
among all functions $\uu=(u_1,\cdots,u_m)\in W^{1,2}(D;\R^m)$, $m\ge 1$, satisfying $\uu=\g$ on $\p D$, where $\g:\p B\to\R^m$ is a prescribed boundary data. This problem can be viewed as  a  vector-valued version of the classical \emph{obstacle problem}. The energy minimizers are solutions of the singular system 
 \begin{align}
    \label{eq:sol}
    \Delta \uu=\frac{\uu}{|\uu|}\rchi_{\{|\uu|>0\}},
\end{align}
where $\rchi_{\{|\uu|>0\}}$ is the characteristic function of $\{|\uu|>0\}$. The regularity of both solutions $\uu$ of \eqref{eq:sol} and their free boundaries $\p\{|\uu|>0\}$ was established in \cite{AndShaUraWei15}, with the help of monotonicity formulas and epiperimetric inequalities.

%%%%%%%%%%%%%%%%%%%%%%%%%%%%%%%%%%%%%%%%%%%%%%%%%%

\subsection{Almost minimizers} Given $r_0>0$, we say that a function $\omega:[0,r_0)\to[0,\infty)$ is a \emph{gauge function} or a modulus of continuity if $\omega$ is monotone nondecreasing and $\omega(0+)=0$.

\begin{definition}\label{def:alm-min}
Let $r_0>0$ be a constant and $\omega(r)$ be a gauge function. We say that $\uu\in W^{1,2}(B_1;\R^m)$ is an almost minimizer for the functional $\int_D(|\D\uu|^2+2|\uu|)$, with gauge function $\omega(r)$, if for any ball $B_r(x_0)\Subset D$ with $0<r<r_0$ and for any $\vv\in \uu+W^{1,2}_0(B_r(x_0);\R^m)$ we have 
$$
\int_{B_r(x_0)}\left(|\D\uu|^2+2|\uu|\right)dx\le (1+\omega(r))\int_{B_r(x_0)}\left(|\D\vv|^2+2|\vv|\right)dx .
$$
\end{definition}
These almost minimizers can be seen as perturbations of minimizers, or solutions of \eqref{eq:sol}. Some examples of almost minimizers can be found in Appendix~\ref{sec:ex}.

The notion of almost minimizer was first introduced for the standard Dirichlet energy functional in \cite{Anz83}. Recently, almost minimizers for Alt-Caffarelli-type functionals were studied in \cite{DavTor15}, \cite{DavEngTor17}, \cite{DesSav20a}, \cite{DesSav20b} etc., and for the thin obstacle problems were considered in \cite{JeoPet20}, \cite{JeoPet21}, \cite{JeoPetSVG20}.

In this paper, we are interested in the regularity of almost minimizers of \eqref{functional} as well as the analysis of their free boundary.

%%%%%%%%%%%%%%%%%%%%%%%%%%%%%%%%%%%%%%%%%%%%%%%%%%

\subsection{Main results}\label{subsec:main-result}

Due to the technical nature of the problem, we restrict ourselves to the case when the gauge function $\omega(r)=r^\al$, $0<\al<2$. Because we are concerned with local properties of almost minimizers and free boundaries, we simply assume that $D$ is the unit ball $B_1$ and the constant $r_0=1$ in Definition~\ref{def:alm-min}.\\

We now state our first main result. Here and henceforth we use notation from Subsection \ref{N}.

\begin{theorem}[Regularity of almost minimizers]\label{thm:grad-u-holder}
Let $\uu\in W^{1,2}(B_1;\R^m)$ be an almost minimizer of \eqref{functional} in $B_1$. Then $\uu\in C^{1,\al/2}(B_1)$. Moreover, for any $K\Subset B_1$, $$
\|\uu\|_{C^{1,\al/2}(K)}\le C(n,\al,K)(E(\uu,1)^{1/2}+1).
$$
\end{theorem}

Next, we consider a class of \emph{half-space} solutions 
\begin{equation*}\Hf := \{ \frac{\max(x\cdot \nu,0)^2}{2} \mathbf{e} : 
\nu \textrm{ is a unit vector in }\R^n \textrm{ and }
\mathbf{e} \textrm{ is a unit vector in }\R^m\}.\end{equation*}
Members of $\Hf$ are $2$-homogeneous global solutions of \eqref{eq:sol}. For 
\begin{equation*}
  M(\vv):= \int_{B_1} \left(|\nabla \vv|^2 + 2 |\vv|\right) -  2 \int_{\partial B_1} |\vv|^2,
\end{equation*}
we define 
\begin{equation*}\frac{\be_n}{2}:= M\left(\frac{\max(x\cdot \nu,0)^2}{2} \mathbf{e}\right).\end{equation*}

Now, we state a Weiss-type monotonicity formula for almost minimizers, which plays a significant role in the analysis of the free boundary.

\begin{theorem}[Weiss-type monotonicity formula]\label{thm:Weiss}
Let $\uu$ be an almost minimizer of \eqref{functional} in $B_1$. For $x_0\in B_{1/2}$, set $$
W(\uu,x_0,t):=\frac{e^{at^\al}}{t^{n+2}}\left[\int_{B_t(x_0)}\left(|\D\uu|^2+2|\uu|\right)-\frac{2(1-bt^\al)}t\int_{\p B_t(x_0)}|\uu|^2\right],
$$
with $$
a=\frac{n+2}\al,\quad b=\frac{n+4}\al.
$$
Then $W(\uu,x_0,t)$ is nondecreasing in $t$ for $0<t<t_0$. Moreover, if $x_0\in \Gamma(\uu)\cap B_{1/2}$ is a free boundary point, then we have either \begin{align}\label{eq:Weiss-limit-classify}
W(\uu,x_0,0+)= \be_n/2\quad\text{or}\quad W(\uu,x_0,0+)\ge \bar{\be}_n
\end{align}
for some $\bar{\be}_n>\be_n/2$.
\end{theorem}

We are now able to define \emph{regular free boundary points}.

\begin{definition}\label{def:reg-set}
We say that a free boundary point $x_0\in \Gamma(\uu)$ is regular if $$W(\uu,x_0,0+)=\be_n/2.$$
We denote by $\mathcal{R}_\uu$ the set of all regular free boundary points of $\uu$.
\end{definition}

Our main result concerning the regularity of the free boundary is as follows.

\begin{theorem}[Regularity of the regular set]\label{thm:reg-set}
$\mathcal{R}_\uu$ is a relatively open subset of the free boundary $\Gamma(\uu)$ and locally a $C^{1,\gamma}$-manifold for some $\gamma=\gamma(n,\al,\kappa)>0$, where $\kappa$ is a constant as in Theorem~\ref{epi}.
\end{theorem}

Our proof is based on the Weiss-type monotonicity formula above and the epiperimetric inequality for solutions of \eqref{eq:sol}. The proof uses a procedure similar to the one for solutions developed in \cite{AndShaUraWei15}.
%%%%%%%%%%%%%%%%%%%%%%%%%%%%%%%%%%%%%%%%%%%%%%%%

\subsection{Notation}\label{N}
Throughout this paper $\R^n$ will be equipped with the Euclidean inner product $x\cdot y$ and the induced norm $|x|$. For $x_0\in\R^n$ and $r>0$, $B_r(x_0)$ means the open $n$-dimensional ball of radius $r$, centered at $x_0$ with boundary $\p B_r(x_0)$. We typically drop the center from the notation if it is the origin.\\
\indent When we consider a given set, we indicate by $\nu$ the unit outward normal vector to the boundary and by $\p_\theta \uu:=\D \uu-(\D \uu\cdot\nu)\nu$ the surface derivative of a given function $\uu$.\\
\indent In integrals, we usually drop the variable and the measure of the integration if it is with respect to the Lebesgue measure or the surface measure. For example, $$
\int_{B_r}\uu=\int_{B_r}\uu(x)\,dx,\quad\int_{\p B_r}\uu=\int_{\p B_r}\uu(x)\,dS_x,
$$
where $S_x$ stands for the surface measure.\\
\indent For a domain $D\subset \R^n$ and an integrable function $\uu$, we denote its integral mean value by $$
\mean{\uu}_D:=\fint_D\uu=\frac1{|D|}\int_D\uu.
$$
In particular, when $D=B_r(x_0)$, we simply write $$
\mean{\uu}_{x_0,r}:=\mean{\uu}_{B_r(x_0)}.
$$
For $\uu\in W^{1,2}(B_r)$, we set
$$
E(\uu,r):=\int_{B_r}\left(|\D\uu|^2+2|\uu|\right).
$$
If $\uu\in W^{1,2}(B_1)$, $x_0\in B_{1/2}$ and $0<r<1/2$, we denote the $2$-homogeneous recaling of $\uu$ by $$
\uu_{x_0,r}(x):=\frac{\uu(x_0+rx)}{r^2},\quad x\in B_{1/(2r)},
$$
and the $2$-homogeneous replacement of $\uu$ in $B_r$ (or equivalently $\uu_{x_0,r}$ in $B_1$) by $$
\cc_{x_0,r}:=|x|^2\uu_{x_0,r}\left(\frac{x}{|x|}\right)=\frac{|x|^2}{r^2}\uu\left(x_0+\frac{r}{|x|}x\right),\quad x\in \R^n.
$$
When $x_0=0$, we simply write $\uu_r$ and $\cc_r$
 for $\uu_{0,r}$ and $\cc_{0,r}$, respectively.

\noindent We also  indicate by $\Gamma(\uu):=\p\{|\uu|>0\}\cap\{\D\uu=0\}$ the free boundary.

\noindent Lastly, it may be worth noting that in the body of a proof, a constant may change from line to line and the dependence of the parameter of the problem is explicitly noted.

%%%%%%%%%%%%%%%%%%%%%%%%%%%%%%%%%%%%%

\section{Regularity of almost minimizers}

In this section, we follow the argument in \cite{Anz83} to obtain the regularity of almost minimizers.

\begin{proposition}\label{prop:Mor-est}
Let $\uu$ be an almost minimizer in $B_1$. Then, there is $C_0=C_0(n)>1$ such that \begin{align}
    \label{eq:alm-min-alm-sub-mean-prop}
    \int_{B_\rho(x_0)}\left(|\D\uu|^2+|\uu|\right)\le C_0\left[\left(\frac\rho r\right)^n+r^\al\right]\int_{B_r(x_0)}\left(|\D\uu|^2+|\uu|\right)+C_0r^{n+2}
\end{align}
whenever $ B_\rho(x_0)\subset B_{r}(x_0)\subset B_1$.
\end{proposition}

\begin{proof}
Let $\h\in W^{1,2}(B_r(x_0);\R^m)$ be such that $\Delta\h=\mathbf0$ in $B_r(x_0)$ and $v=u$ on $\p B_r(x_0)$. Since $|\h|$ and $|\D \h|^2$ are subharmonic, they satisfy the sub-mean value properties \begin{align}
    \label{eq:sub-mean-prop}
    \int_{B_\rho(x_0)}|\h|\le \left(\frac\rho r\right)^n\int_{B_r(x_0)}|\h|,\quad \int_{B_\rho(x_0)}|\D\h|^2\le \left(\frac\rho r\right)^n\int_{B_r(x_0)}|\D\h|^2.
\end{align}
Moreover, it follows from the weak equation $\int_{B_r(x_0)}\D\h\D(\uu-\h)=0$ and the almost minimizing property of $\uu$ that \begin{align}
    \label{eq:grad-(u-h)-est}\begin{split}
    \int_{B_r(x_0)}|\D(\uu-\h)|^2&=\int_{B_r(x_0)}|\D\uu|^2-|\D\h|^2\\
    &\le \int_{B_r(x_0)}r^\al|\D\h|^2+2(1+r^\al)|\h|-2|\uu|\\
    &=\int_{B_r(x_0)}r^\al|\D\h|^2+2(1+r^\al)(|\h|-|\uu|)+2r^\al|\uu|\\
    &\le \int_{B_r(x_0)}r^\al|\D\uu|^2+2(1+r^\al)|\h-\uu|+2r^\al|\uu|,
\end{split}\end{align}
where in the last inequality we used the fact that $\h$ is the energy minimizer of the Dirichlet integral $\int_{B_r(x_0)}|\D\h|^2$.
We also apply Poincar\`e inequality and Young's inequality to obtain
\begin{align*}
    5\int_{B_r(x_0)}|\uu-\h|&\le C(n)r\int_{B_r(x_0)}|\D(\uu-\h)|\le \frac12\int_{B_r(x_0)}|\D(\uu-\h)|^2+C(n)r^{n+2}.
\end{align*}
From this and \eqref{eq:grad-(u-h)-est}, we have \begin{align*}
    &\int_{B_r(x_0)}\left(|\D(\uu-\h)|^2+5|\uu-\h|\right)\\
    &\le \int_{B_r(x_0)}\left(\frac12|\D(\uu-\h)|^2+2(1+r^\al)|\uu-\h|+r^\al|\D\uu|^2+2r^\al|\uu|\right)+C(n)r^{n+2},
\end{align*}
From $B_r(x_0)\subset B_1$, we see that $0<r<1$, thus $2(1+r^\al)\le4$. Then, it follows that
\begin{align}
    \label{eq:(u-h)-diff-est}
    \int_{B_r(x_0)}\left(|\D(\uu-\h)|^2+|\uu-\h|\right)\le C(n)r^\al\int_{B_r(x_0)}\left(|\D\uu|^2+|\uu|\right)+C(n)r^{n+2}.
\end{align}
Now, by combining \eqref{eq:sub-mean-prop} and \eqref{eq:(u-h)-diff-est}, we obtain that for $0<\rho<r<r_0$, \begin{align*}
    &\int_{B_\rho(x_0)}\left(|\D \uu|^2+|\uu|\right)\\
    &\qquad\le 2\int_{B_\rho(x_0)}\left(|\D\h|^2+|\D(\uu-\h)|^2+|\h|+|\uu-\h|\right)\\
    &\qquad\le 2\left(\frac\rho r\right)^n\int_{B_r(x_0)}\left(|\D\h|^2+|\h|\right)+2\int_{B_\rho(x_0)}\left(|\D(\uu-\h)|^2+|\uu-\h|\right)\\
    &\qquad\le 4\left(\frac\rho r\right)^n\int_{B_r(x_0)}\left(|\D\uu|^2+2|\uu|+|\D(\uu-\h)|^2+2|\uu-\h|\right)\\
    &\qquad\qquad+2\int_{B_\rho(x_0)}\left(|\D(\uu-\h)|^2+|\uu-\h|\right)\\
    &\qquad\le 4\left(\frac\rho r\right)^n\int_{B_r(x_0)}\left(|\D\uu|^2+2|\uu|\right)+10\int_{B_r(x_0)}\left(|\D(\uu-\h)|^2+|\uu-\h|\right)\\
    &\qquad\le C(n)\left[\left(\frac\rho r\right)^n+r^\al\right]\int_{B_r(x_0)}\left(|\D\uu|^2+|\uu|\right)+C(n)r^{n+2}.\qedhere
\end{align*}
\end{proof}

From here, we deduce the almost Lipschitz regularity of $\uu$ with the help of the following lemma, whose proof can be found in \cite{HanLin97}.

\begin{lemma}\label{lem:HL}
  Let $r_0>0$ be a positive number and let
  $\vp:(0,r_0)\to (0, \infty)$ be a nondecreasing function. Let $a$,
  $\beta$, and $\gamma$ be such that $a>0$, $\gamma >\beta >0$. There
  exist two positive numbers $\e=\e(a,\gamma,\beta)$,
  $c=c(a,\gamma,\beta)$ such that, if
$$
\vp(\rho)\le
a\Bigl[\Bigl(\frac{\rho}{r}\Bigr)^{\gamma}+\e\Bigr]\vp(r)+b\, r^{\be}
$$ for all $\rho$, $r$ with $0<\rho\leq r<r_0$, where $b\ge 0$,
then one also has, still for $0<\rho<r<r_0$,
$$
\vp(\rho)\le
c\Bigl[\Bigl(\frac{\rho}{r}\Bigr)^{\be}\vp(r)+b\rho^{\be}\Bigr].
$$
\end{lemma}

\begin{theorem}\label{thm:alm-Lip-reg}
Let $\uu$ be an almost minimizer in $B_1$. Then $\uu\in C^{0,\si}(B_1)$ for all $0<\si<1$. Moreover, for any $K\Subset B_1$, $$
\|\uu\|_{C^{0,\si}(K)}\le C\left(E(\uu,1)^{1/2}+1\right)
$$
with $C=C(n,\al,\si,K)$.
\end{theorem}

\begin{proof}
For given $K\Subset B_1$ and $x_0\in K$, take $\de=\de(n,\al,\si,K)>0$ such that $\de<\dist(K,\p B_1)$ and $\de^\al\le \e(C_0,n,n+2\si-2)$, where $C_0=C_0(n)$ is as in Proposition~\ref{prop:Mor-est} and $\e=\e(C_0,n,n+2\si-2)$ is as in Lemma~\ref{lem:HL}. Then, by \eqref{eq:alm-min-alm-sub-mean-prop}, for $0<\rho<r<\de$, \begin{align*}
    \int_{B_\rho(x_0)}\left(|\D\uu|^2+|\uu|\right)\le C_0\left[\left(\frac\rho r\right)^n+\e\right]\int_{B_r(x_0)}\left(|\D\uu|^2+|\uu|\right)+C_0r^{n+2\si-2}.
\end{align*}
By applying Lemma~\ref{lem:HL}, we obtain \begin{align*}
    \int_{B_\rho(x_0)}\left(|\D \uu|^2+|\uu|\right)\le C(n,\si)\left[\left(\frac\rho r\right)^{n+2\si-2}\int_{B_r(x_0)}\left(|\D\uu|^2+|\uu|\right)+\rho^{n+2\si-2}\right] .
\end{align*}
Taking $r\nearrow\de(n,\al,\si,K)$, we get \begin{align}\label{eq:alm-min-Morrey-est}
\int_{B_\rho(x_0)}\left(|\D\uu|^2+|\uu|\right)\le C(n,\al,\si,K)\left(E(\uu,1)+1\right)\rho^{n+2\si-2} ,
\end{align}
for $0<\rho<\de$. In particular, we have $$
\int_{B_\rho(x_0)}|\D\uu|^2\le C(n,\al,\si,K)\left(E(\uu,1)+1\right)\rho^{n+2\si-2},
$$
and by Morrey space embedding we conclude $\uu\in C^{0,\si}(K)$ with \begin{equation*}
\|\uu\|_{C^{0,\si}(K)}\le C(n,\al,\si,K)\left(E(\uu,1)^{1/2}+1\right).\qedhere
\end{equation*}
\end{proof}

Now we prove Theorem~\ref{thm:grad-u-holder} using the above almost Lipschitz estimate of almost minimizers.

\begin{proof}[Proof of Theorem~\ref{thm:grad-u-holder}]
For $K\Subset B_1$, fix a small $r_0=r_0(n,\al,K)>0$ to be chosen later. Particularly, we ask $r_0<\dist(K,\p B_1)$. For $x_0\in K$ and $0<r<r_0$, let $\h\in W^{1,2}(B_r(x_0);\R^m)$ be a harmonic function such that $\h=\uu$ on $\p B_r(x_0)$. Then, by \eqref{eq:(u-h)-diff-est} and \eqref{eq:alm-min-Morrey-est} with $\si=1-\al/4\in (0,1)$, \begin{align}\begin{split}\label{eq:grad-(u-h)-est-2}
    \int_{B_r(x_0)}|\D(\uu-\h)|^2&\le C(n)r^\al\int_{B_r(x_0)}\left(|\D\uu|^2+|\uu|\right)+C(n)r^{n+2}\\
    &\le C(n,\al,K)\left(E(\uu,1)+1\right)r^{n+\al/2}+C(n)r^{n+2}\\
    &\le C(n,\al,K)\left(E(\uu,1)+1\right)r^{n+\al/2} ,
\end{split}\end{align}
for $0<r<r_0(n,\al,K)$. Note that since $\h$ is harmonic in $B_r(x_0)$, for $0<\rho<r$ $$
\int_{B_\rho(x_0)}|\D\h-\mean{\D\h}_{x_0,\rho}|^2\le \left(\frac\rho r\right)^{n+2}\int_{B_r(x_0)}|\D\h-\mean{\D\h}_{x_0,r}|^2.
$$
Moreover, by Jensen's inequality, \begin{align*}
    &\int_{B_\rho(x_0)}|\D\uu-\mean{\D\uu}_{x_0,\rho}|^2\\
    &\qquad\le 3\int_{B_\rho(x_0)}|\D\h-\mean{\D\h}_{x_0,\rho}|^2+|\D(\uu-\h)|^2+|\mean{\D(\uu-\h)}_{x_0,\rho}|^2\\
    &\qquad\le 3\int_{B_\rho(x_0)}|\D\h-\mean{\D\h}_{x_0,\rho}|^2+6\int_{B_\rho(x_0)}|\D(\uu-\h)|^2,
\end{align*}
and similarly, $$
\int_{B_r(x_0)}|\D\h-\mean{\D\h}_{x_0,r}|^2\le 3\int_{B_r(x_0)}|\D \uu-\mean{\D\uu}_{x_0,r}|^2+6\int_{B_r(x_0)}|\D(\uu-\h)|^2.
$$
Now, we use the above inequalities to obtain \begin{align*}
    &\int_{B_\rho(x_0)}|\D\uu-\mean{\D\uu}_{x_0,\rho}|^2\\
    &\qquad\le 3\int_{B_\rho(x_0)}|\D\h-\mean{\D\h}_{x_0,\rho}|^2+6\int_{B_\rho(x_0)}|\D(\uu-\h)|^2\\
    &\qquad\le 3\left(\frac\rho r\right)^{n+2}\int_{B_r(x_0)}|\D\h-\mean{\D\h}_{x_0,r}|^2+6\int_{B_\rho(x_0)}|\D(\uu-\h)|^2\\
    &\qquad\le 9\left(\frac\rho r\right)^{n+2}\int_{B_r(x_0)}|\D \uu-\mean{\D \uu}_{x_0,r}|^2+24\int_{B_r(x_0)}|\D(\uu-\h)|^2\\
    &\qquad\le 9\left(\frac\rho r\right)^{n+2}\int_{B_r(x_0)}|\D\uu-\mean{\D\uu}_{x_0,r}|^2+C(n,\al,K)\left(E(\uu,1)+1\right)r^{n+\al/2}.
\end{align*}
Thus, applying Lemma~\ref{lem:HL} we get \begin{align*}
    \int_{B_\rho(x_0)}|\D\uu-\mean{\D\uu}_{x_0,\rho}|^2&\le C(n,\al)\left(\frac\rho r\right)^{n+\al/2}\int_{B_r(x_0)}|\D\uu-\mean{\D\uu}_{x_0,r}|^2\\
    &\qquad+C(n,\al,K)\left(E(\uu,1)+1\right)\rho^{n+\al/2} ,
\end{align*}
for $0<\rho<r<r_0$. Taking $r\nearrow r_0(n,\al,K)$, we have \begin{align*}
\int_{B_\rho(x_0)}|\D\uu-\mean{\D\uu}_{x_0,\rho}|^2\le C(n,\al,K)\left(E(\uu,1)+1\right)\rho^{n+\al/2}.
\end{align*}
By Campanato space embedding, we obtain $\D\uu\in C^{0,\al/4}(K)$ with $$
\|\D\uu\|_{C^{0,\al/4}(K)}\le C(n,\al,K)(E(\uu,1)^{1/2}+1).
$$
In particular, we have $$
\|\D\uu\|_{L^\infty(K)}\le C(n,\al,K)(E(\uu,1)^{1/2}+1) ,
$$
for any $K\Subset B_1$. With this estimate, we can improve \eqref{eq:grad-(u-h)-est-2} \begin{align*}
    \int_{B_r(x_0)}|\D(\uu-\h)|^2&\le C(n)r^\al\int_{B_r(x_0)}\left(|\D\uu|^2+|\uu|\right)+C(n)r^{n+2}\\
    &\le C(n,\al,K)\left(E(\uu,1)+1\right)r^{n+\al}+C(n)r^{n+2}\\
    &\le C(n,\al,K)\left(E(\uu,1)+1\right)r^{n+\al},
\end{align*}
and by repeating the process above we conclude that $\D\uu\in C^{1,\al/2}(K)$ with \begin{equation*}
\|\D\uu\|_{C^{0,\al/2}(K)}\le C(n,\al,K)(E(\uu,1)^{1/2}+1).\qedhere
\end{equation*}
\end{proof}

%%%%%%%%%%%%%%%%%%%%%%%%%%%%%%%%%%%%%%%%%%%%

\section{Weiss-type monotonicity formula}
In the rest of the paper, we study the free boundary of almost minimizers. This section is devoted to proving the monotonicity of the Weiss-type functional introduced in Theorem~\ref{thm:Weiss}. It is obtained by comparing almost minimizers with homogeneous functions of degree $2$ with some algebraic manipulation. This argument goes back to \cite{Wei99} in the case of the classical obstacle problem, and recently to \cite{JeoPet21} for almost minimizers of the thin obstacle problem.

\begin{theorem}\label{thm:weiss-mon}
Let $\uu$ be an almost minimizer in $B_1$. Then, for $x_0\in B_{1/2}$ and $0<t<t_0(n,\al)$, $$
\frac{d}{dt}W(\uu,x_0,t)\ge\frac{e^{at^\al}}{t^{n+2}}\int_{\p B_t(x_0)}\left|\p_\nu\uu-\frac{2(1-bt^\al)}t\uu\right|^2.
$$
In particular, $W(\uu,x_0,t)$ is nondecreasing in $t$ for $0<t<t_0$.
\end{theorem}

\begin{proof}The proof can be obtained with similar arguments as in \cite{JeoPet21}, and we sketch it below. Assume without loss of generality that $x_0=0.$ Then, for the $2$-homogeneous replacement $\cc_r$ of $\uu$ in $B_r$,
a standard computation (see Theorem 5.1 in \cite{JeoPet21}) gives that
$$\int_{B_r} |\nabla \cc_r|^2 =  \frac{r}{n+2} \int_{\p B_r} \left(|\nabla \uu|^2 - |\p_\nu \uu|^2 + \frac{4}{r^2} |\uu|^2\right).$$ Using that $\uu$ is almost minimizing and $\uu=\cc_r$ on $\p B_r,$ we deduce that,
\begin{align*}\int_{\p B_r} |\uu|^2 &= \frac{r^2}{4}\left(\frac{n+2}{r}E(\cc_r,r) - 2 \frac{n+2}{r}\int_{B_r}|\cc_r| - \int_{\p B_r}\left(|\nabla \uu|^2 - |\p_\nu \uu|^2\right)\right)\\
& \geq \frac{r^2}{4}\left(\frac{n+2}{r}(1-r^{\alpha})E(\uu,r) - 2 \frac{n+2}{r}\int_{B_r}|\cc_r| - \int_{\p B_r}\left(|\nabla \uu|^2 - |\p_\nu \uu|^2\right)\right) \\
& \geq \frac{r^2}{4}\left(\frac{n+2}{r}(1-r^{\alpha})E(\uu,r) - 2 \int_{\p B_r}|\uu| - \int_{\p B_r}\left(|\nabla \uu|^2 - |\p_\nu \uu|^2\right)\right),
\end{align*} where in the last line we used the definition of $\cc_r.$ 
From this we deduce that,
\begin{equation}\label{bb}
E(\uu,r) \leq \frac{r}{(n+2)(1-r^\alpha)} \left(\frac{4}{r^2} \int_{\p B_r} |\uu|^2 + 2 \int_{\p B_r}|\uu| + \int_{\p B_r}\left(|\nabla \uu|^2 - |\p_\nu \uu|^2\right)\right).
\end{equation}

We now set
$$\psi(r):= \frac{2e^{ar^\alpha} (1-br^\alpha)}{r^{n+3}},$$ so that
$$W(r):=W(\uu, 0, r)= e^{ar^\alpha}r^{-n-2}E(\uu, r) - \psi(r)\int_{\p B_r}|\uu|^2,$$ 
and
\begin{equation}\label{deriv}
\frac{d}{dr} W(r)=\frac{d}{dr}(e^{ar^\alpha}r^{-n-2})E(\uu,r) + e^{ar^\alpha}r^{-n-2} \int_{\p B_r}(|\nabla \uu|^2+ 2|\uu|) + \Phi(r) ,
\end{equation}
with 
$$\Phi(r):= -\frac{d}{dr}\left (\psi(r)\int_{\p B_r}|\uu|^2\right) .$$
 On the other hand,
$$\frac{d}{dr}(e^{ar^\alpha}r^{-n-2})E(\uu,r)=-(n+2)e^{ar^\alpha}r^{-n-3}(1-r^{\alpha})E(\uu,r),$$
 which in view of \eqref{bb} gives
$$\frac{d}{dr}(e^{ar^\alpha}r^{-n-2})E(\uu,r) \geq e^{ar^\alpha} r^{-n-2}\left(\int_{\p B_r}(|\p_\nu \uu|^2 - |\nabla \uu|^2) -2 \int_{\p B_r} |\uu|\right)$$
$$-4 e^{ar^\alpha} r^{-n-4} \int_{\p B_r} |\uu|^2.$$
This, combined with \eqref{deriv} implies
$$\frac{d}{dr} W(r) \geq e^{ar^\alpha} r^{-n-2}\int_{\p B_r}|\p_\nu \uu|^2 -4 e^{ar^\alpha} r^{-n-4} \int_{\p B_r} |\uu|^2$$
$$-2\psi(r)\int_{\p B_r}\uu \cdot \p_\nu \uu -(4e^{ar^\alpha}r^{-n-4} +\psi'(r) + \frac{n-1}{r}\psi(r) ) \int_{\p B_r} |\uu|^2.$$
It is easy to see that, for $r< r_0=r_0(n,\alpha)$
$$-\frac{e^{ar^\alpha}}{r^{n+2}}(4e^{ar^\alpha}r^{-n-4} +\psi'(r) + \frac{n-1}{r}\psi(r)) \geq \psi(r)^2,$$ which combined with the inequality above and the definition of $\psi$ implies the desired claim.
\end{proof}

%%%%%%%%%%%%%%%%%%%%%%%%%%%%%%%%%%%%%%%%%%%%

\section{Growth estimates}
In this section we show that almost minimizers have quadratic growth away from the free boundary. We begin with a weak quadratic growth estimate.

\begin{lemma}
\label{lem:weak-growth-est}
Let $\uu\in W^{1,2}(B_1)$ be an almost minimizer in $B_1$. Then for any $0<\e<1$, there exist $C>0$ and $r_0>0$, depending on $n$, $\al$, $\e$, $E(\uu,1)$, such that $$
\sup_{B_r(x)}\left(\frac{|\uu|}{r^{2-\e}}+\frac{|\D\uu|}{r^{1-\e}}\right)\le C ,
$$
for any $x\in\Gamma(\uu)\cap B_{1/2}$ and $0<r<r_0.$\end{lemma}

\begin{proof}
Let $M\geq E(\uu,1)$, and note that $$
\|\uu\|_{C^{1,\al/2}(B_{3/4})}\le C(n,\al)(E(\uu,1)^{1/2}+1)\le C(n,\al,M).
$$
Assume by contradiction that the lemma is not true for almost minimizers with energy bounded by $M$. Then we can find a sequence $\{\uu_j\}_{j=1}^\infty$ of almost minimizers with $E(\uu_j,1)\le M$, a sequence $\{x_j\}_{j=1}^\infty\subset \Gamma(\uu_j)\cap B_{1/2}$ of free boundary points and a sequence of positive radii $\{r_j\}_{j=1}^\infty\subset (0,1)$,  $r_j \searrow 0,$
such that for large $j$ \begin{align*}
    \sup_{B_{r_j}(x_j)}\left(\frac{|\uu_j|}{r_j^{2-\e}}+\frac{|\D\uu_j|}{r_j^{1-\e}}\right)=j,\quad \sup_{B_r(x_j)}\left(\frac{|\uu_j|}{r^{2-\e}}+\frac{|\D\uu_j|}{r^{1-\e}}\right)\le j\quad\text{for any } r_j\le r\le 1/4.
\end{align*}
Define the function $$
\wuuj(x):=\frac{\uu_j(r_jx+x_j)}{jr_j^{2-\e}},\quad x\in B_{\frac1{4r_j}}.
$$
Then
 $$
\sup_{B_1}\left(|\wuuj|+|\D\wuuj|\right)=1\quad\text{and } \sup_{B_R}\left(\frac{|\wuuj|}{R^{2-\e}}+\frac{|\D\wuuj|}{R^{1-\e}}\right)\le 1\quad\text{for any }1\le R\le \frac1{4r_j}.
$$
Now we claim that there exists a harmonic function $\wuuz\in C^1_{\loc}(\R^n;\R^m)$ such that over a subsequence $$
\wuuj\to\wuuz\quad\text{in }C^1_{\loc}(\R^n;\R^m).
$$
Indeed, for a fixed $R>1$ and a ball $B_\rho(z)\subset B_R$, we have $$
\int_{B_\rho(z)}\left(|\D\wuuj|^2+\frac{2r_j^\e}j|\wuuj|\right)=\frac1{j^2r_j^{n+2-2\e}}\int_{B_{r_j\rho}(r_jz+x_j)}\left(|\D\uu_j|^2+2|\uu_j|\right).
$$
This implies that each $\wuuj$ is an almost minimizer of a functional $$
J_{B_\rho(z)}(\wuuj)=\int_{B_\rho(z)}\left(|\D\wuuj|^2+\frac{2r_j^\e}{j}|\wuuj|\right) ,
$$
with a gauge function $\omega(\rho)=(r_j\rho)^\al\le\rho^\al$. Let $\h$ be the harmonic replacement of $\wuuj$ in $B_\rho(z)$. Then, since $\frac{r_j^\e}j\le 1$, we can obtain the equivalent of \eqref{eq:grad-(u-h)-est}: \begin{align*}
    \int_{B_\rho(z)}|\D(\wuuj-\h)|
    &=\int_{B_\rho(z)}\left(|\D\wuuj|^2-|\D \h|^2\right)\\
    &\le \int_{B_\rho(z)}\rho^\al|\D \h|^2+\frac{2r_j^\e}{j}\left((1+\rho^\al)|\h|-|\wuuj|\right)\\
    &=\int_{B_\rho(z)}\rho^\al|\D \h|^2+\frac{2r_j^\e}j(1+\rho^\al)(|\h|-|\wuuj|)+\frac{2r_j^\e}j\rho^\al|\wuuj|\\
    &\le \int_{B_\rho(z)}\rho^\al|\D\wuuj|^2+2(1+\rho^\al)|\h-\wuuj|+2\rho^\al|\wuuj|.
\end{align*}
With this estimate at hand, we can proceed as in the proofs of Proposition~\ref{prop:Mor-est}, Theorem~\ref{thm:alm-Lip-reg} and Theorem~\ref{thm:grad-u-holder} to obtain $$
\|\wuuj\|_{C^{1,\al/2}(\overline{B_{R/2}};\R^m)}\le C(n,\al,R)\left(E(\wuuj,R)^{1/2}+1\right)\le C(n,\al,M,R).
$$
Thus, over a subsequence, $$
\wuuj\to\wuuz\quad\text{in }C^1(B_{R/2};\R^m).
$$
Letting $R\to\infty$ and using Cantor's diagonal argument, we have that $$
\wuuj\to\wuuz\quad\text{in }C^1_{\loc}(\R^n;\R^m).
$$
Now, to verify that $\wuuz$ is harmonic, we fix $R>1$ and observe that for large $j$ and for the harmonic replacement $\h_j$ of $\wuuj$ in $B_R$, \begin{align}\label{eq:alm-min-prop}
\int_{B_R}\left(|\D\wuuj|^2+\frac{2r_j^\e}j|\wuuj|\right)\le \left(1+(r_jR)^\al\right)\int_{B_R}\left(|\D\h_j|^2+\frac{2r_j^\e}j|\h_j|\right).
\end{align}
From $$
\|\h_j\|_{C^{1,\al/2}(\overline{B_R};\R^m)}\le C(n,R)\|\wuuj\|_{C^{1,\al/2}(\overline{B_R};\R^m)}\le C(n,\al,M,R),
$$
we have that up to a subsequence $$
\h_j\to \h_0 \quad\text{in }C^1(\overline{B_R};\R^m) ,
$$
for some harmonic function $\h_0\in C^1(\overline{B_R};\R^m)$. Thus, by taking $j\to\infty$ in \eqref{eq:alm-min-prop}, we obtain $$
\int_{B_R}|\D\wuuz|^2\le \int_{B_R}|\D \h_0|^2.
$$
Since $\h_0$ is a harmonic function in $B_R$ with $\wuuz$-boundary value, $\wuuz$ should be equal to $\h_0$ in $B_R$ and hence is harmonic there. Since $R>1$ is arbitrary, $\wuuz$ is harmonic in $\R^n$. This finishes the proof of the claim.\\

\medskip\noindent Now, we observe that \begin{align*}
    \sup_{B_1}(|\wuuz|+|\D\wuuz|)=1 \quad\text{and }\sup_{B_R}\left(\frac{|\wuuz|}{R^{2-\e}}+\frac{|\D\wuuz|}{R^{1-\e}}\right)\le 1\quad\text{for any }R\ge 1.
\end{align*}
Moreover, notice that from $x_j\in\Gamma(\uu_j)$, we have $|\wuuj(0)|=|\D\wuuj(0)|=0$, and thus $|\wuuz(0)|=|\D\wuuz(0)|=0$. To arrive at a contradiction, we note that $|\D\wuuz(x)|\le |x|^{1-\e}$ for $|x|\ge 1$ and $|\D\wuuz(0)|=0$ and apply Liouville's theorem to obtain $\D\wuuz\equiv 0$ in $\R^n$. This means that $\wuuz$ is a constant vector, which contradicts that $\sup_{B_1}|\wuuz|=\sup_{B_1}(|\wuuz|+|\D\wuuz|)=1$ and $|\wuuz(0)|=0$.
\end{proof}

Now we are going to establish the quadratic growth of almost minimizers at free boundary points (Lemma~\ref{lem:quad-growth-L1} and Lemma~\ref{lem:quad-growth-L2}) with the help of the weak quadratic growth estimate (Lemma~\ref{lem:weak-growth-est}) and the monotonicity of the Weiss-type energy (Theorem~\ref{thm:weiss-mon}).

\begin{lemma}\label{lem:quad-growth-L1}
Let $\uu$ be an almost minimizer in $B_1$ with $x_0\in \Gamma(\uu)\cap B_{1/2}$. Then, for $C>0$ and $r_0>0$, depending only on $n$, $\al$, $E(\uu,1)$, we have $$
\int_{B_r(x_0)}|\uu|\le Cr^{n+2},\quad 0<r<r_0.
$$
\end{lemma}

\begin{proof}
From the monotonicity formula, $$
W(\uu,x_0,r)\le W(\uu,x_0,1/2)\le C(n,\al)E(u,1).
$$
Thus \begin{align}\label{eq:L1-bound}\begin{split}
    &\frac{2e^{ar^\al}}{r^{n+2}}\int_{B_r(x_0)}|\uu|\\
    &=W(\uu,x_0,r)+\frac{2e^{ar^\al}(1-br^\al)}{r^{n+3}}\int_{\partial B_r(x_0)}|\uu|^2-\frac{e^{ar^\al}}{r^{n+2}}\int_{B_r(x_0)}|\D\uu|^2\\
    &\le C(n,\al)E(\uu,1)+e^{ar^\al}\left(\frac2{r^{n+3}}\int_{\p B_r(x_0)}|\uu|^2-\frac1{r^{n+2}}\int_{B_r(x_0)}|\D\uu|^2\right)\\
    &=C(n,\al)E(\uu,1)\\
    &\qquad+e^{ar^\al}\left(\frac2{r^{n+3}}\int_{\partial B_r(x_0)}|\uu-S_{x_0}\pp|^2-\frac1{r^{n+2}}\int_{B_r(x_0)}|\D(\uu-S_{x_0}\pp)|^2\right)\\
    &\le C(n,\al)E(\uu,1)+\frac{2e^{ar^\al}}{r^{n+3}}\int_{\p B_r(x_0)}|\uu-S_{x_0}\pp|^2 
\end{split}\end{align}
for each $\pp=(p_1,\ldots,p_m)\in \cH$, where $\cH$ is the set of all $\pp=(p_1,\ldots,p_m)$ such that each component $p_j$ is $2$-homogeneous harmonic polynomial, and $S_{x_0}f(x):=f(x-x_0)$.\\
For each $x_0\in \Gamma(\uu)$, let $\pp_{x_0,r}=\pp_{x_0,r,\uu}$ be the minimizer of $\int_{\p B_r(x_0)}|\uu-S_{x_0}\pp|^2$ in $\cH$. Then it is easy to see that $$
\int_{\partial B_r(x_0)}(\uu-S_{x_0}\pp_{x_0,r})S_{x_0}\q=0\quad\text{for every }\q\in\cH.
$$
Now, the lemma follows once we prove that there are $C=C(n,\al,E(\uu,1))$ and $r_0=r_0(n,\al, E(\uu,1))>0$ such that \begin{align}\label{eq:u-p-diff-est}
\int_{\p B_r(x_0)}|\uu-S_{x_0}\pp_{x_0,r}|^2\le Cr^{n+3},\quad0<r<r_0.
\end{align}
Towards this, we assume to the contrary that there are a sequence of almost minimizers $\uu_k$, a sequence of points $x_k\in \Gamma(\uu_k)\cap B_{1/2}$, and $r_k\to 0$ such that $E(\uu_k,1)$ is uniformly bounded and $$
M_k:=\frac1{r_k^{n+3}}\int_{\partial B_{r_k}(x_k)}|\uu_k-S_{x_k}\pp_{x_k,r_k}|^2\to\infty.
$$
Denote $$
\vv_k(x):=(\uu_k)_{x_k,r_k}(x)=\frac{\uu_k(x_k+r_kx)}{r_k^2}\quad\text{and }\w_k(x):=\frac{\vv_k(x)-\pp_{x_k,r_k}(x)}{M_k^{1/2}},\quad x\in B_1.
$$
For each $k$, let $\h_k$ be the harmonic function in $B_1$ such that $\h_k=\vv_k$ on $\partial B_1$, and write $$
\w_k=\frac{\h_k-\pp_{x_k,r_k}}{M_k^{1/2}}+\frac{\vv_k-\h_k}{M_k^{1/2}}:=\w_k^1+\w_k^2.
$$
Then  \begin{align*}
    \int_{\p B_1}|\w_k^1|^2=\frac1{M_k}\int_{\p B_1}|\vv_k-\pp_{x_k,r_k}|^2=\frac1{M_kr_k^{n+3}}\int_{\p B_{r_k}(x_k)}|\uu_k-S_{x_k}\pp_{x_k,r_k}|^2=1
\end{align*}
and
 \begin{align*}
   \int_{B_1}|\D\w_k^1|^2-2\int_{\p B_1}|\w_k^1|^2
    &=\frac1{M_k}\left(\int_{B_1}|\D(\h_k-\pp_{x_k,r_k})|^2-2\int_{\p B_1}|\h_k-\pp_{x_k,r_k}|^2\right)\\
    &=\frac1{M_k}\left(\int_{B_1}|\D\h_k|^2-2\int_{\p B_1}|\h_k|^2\right)\\
    &\le\frac1{M_k}\left(\int_{B_1}|\D\vv_k|^2-2\int_{\p B_1}|\vv_k|^2\right)\\
    &=\frac1{M_k}\left(\frac1{r_k^{n+2}}\int_{B_{r_k}(x_k)}|\D\uu_k|^2-\frac2{r_k^{n+3}}\int_{\p B_{r_k}(x_k)}|\uu_k|^2\right)\\
    &\le \frac1{M_k}W(\uu_k,x_k,r_k)\le \frac1{M_k}W(\uu_k,x_k,1/2) \\
    &\le \frac{C(n,\al)}{M_k}E(\uu_k,1)\to 0\quad\text{as }k\to\infty.
\end{align*}
Thus $\{\w_k^1\}_{k\in\N}$ is bounded in $W^{1,2}(B_1;\R^m)$, hence there is $\w_0^1\in W^{1,2}(B_1;\R^m)$ such that up to a subsequence \begin{align*}
    &\w_k^1\to\w_0^1\quad\text{weakly in }W^{1,2}(B_1;\R^m),\\
    &\w_k^1\to\w_0^1\quad\text{strongly in }L^2(\p B_1;\R^m).
\end{align*}
It follows that \begin{align}\label{eq:w_0^1}
&\int_{B_1}|\D\w_0^1|^2\le 2\int_{\p B_1}|\w_0^1|^2=2,\\
\label{eq:w_0^1-q-est}&\int_{\p B_1}\w_0^1\cdot\q=0\quad\text{for every }\q\in\cH.
\end{align}
Moreover, since each $\w_k^1$ is harmonic in $B_1$, $\w_0^1$ is also harmonic in $B_1$. It also follows from $C^{1,\al}$-estimate of harmonic functions $\w_k^1$ that $$
\w_k^1\to\w_0^1\quad\text{in }C^1_{\loc}(B_1;\R^m).
$$
Next, to deal with $\w_k^2=\frac{\vv_k-\h_k}{M_k^{1/2}}$, we observe that $\vv_k(x)=\frac{\uu_k(x_0+r_kx)}{r_k^2}$ is an almost minimizer with a gauge function $\omega(\rho)=(\rho r_k)^\al$. Thus, by \eqref{eq:(u-h)-diff-est} and Lemma~\ref{lem:weak-growth-est} with $\e=\al/4$, \begin{align*}
    \int_{B_1}\left(|\D\w_k^2|^2+|\w_k^2|\right)&\le\frac1{M_k^{1/2}}\int_{B_1}\left(|\D(\vv_k-\h_k)|^2+|\vv_k-\h_k|\right)\\
    &\le \frac{C(n)}{M_k^{1/2}}\left(r_k^\al\int_{B_1}\left(|\D\vv_k|^2+|\vv_k|\right)+1\right)\\
    &\le \frac{C(n)}{M_k^{1/2}}\left(C(n,\al,E(\uu_k,1))r_k^{\al/2}+1\right)\to 0\quad\text{as }k\to\infty.
\end{align*}
Thus, up to another subsequence, $$
\w_k^2\to0\quad\text{strongly in }W^{1,2}(B_1;\R^m).
$$
From $x_k\in\Gamma(\uu_k)$, we have $|\w_k(0)|=|\D\w_k(0)|=0$, thus \begin{align*}
    |\w_0^1(0)|=\lim_{k\to\infty}|\w_k^1(0)|=\lim_{k\to\infty}|\w_k(0)|=0\end{align*}
    and
    \begin{align*}|\D\w_0^1(0)|=\lim_{k\to\infty}|\D\w_k^1(0)|=\lim_{k\to\infty}|\D\w_k(0)|=0.
\end{align*}
From here, we can repeat the argument in Theorem~2 in \cite{AndShaUraWei15} to get the contradiction. Indeed, by Lemma 4.1 in \cite{Wei01}, each component $z_j$ of $\w_0^1$ satisfies $$
2\int_{\p B_1}z_j^2\le \int_{B_1}|\D z_j|^2.
$$
This, together with \eqref{eq:w_0^1}, gives $$
2\int_{\p B_1}z_j^2= \int_{B_1}|\D z_j|^2.
$$
By Lemma 4.1 in \cite{Wei01} again, each $z_j$ is a $2$-homogeneous harmonic polynomial. Therefore, we have $\w_0^1\in \cH$, and this contradicts \eqref{eq:w_0^1-q-est}.
\end{proof}

\begin{lemma}
\label{lem:quad-growth-L2}
Let $\uu$ be an almost minimizer in $B_1$ with $x_0\in \Gamma(\uu)\cap B_{1/2}$. Then, for $C$ and $r_0$ as in Lemma~\ref{lem:quad-growth-L1}, we have for $0<r<r_0$ 
\begin{align*}
    \int_{B_r(x_0)}|\D\uu|^2 \le Cr^{n+2} \qquad   \hbox{and} \qquad    \int_{\p B_r(x_0)}|\uu|^2\le Cr^{n+3}.
\end{align*}
\end{lemma}

\begin{proof}
For simplicity we assume $x_0=0$. For $0<r<r_0$, let $\h$ be a harmonic replacement of $\uu$ in $B_r$. For $2$-homogeneous recalings $\uu_r$ and $\h_r$ of $\uu$ and $\h$, respectively, by \eqref{eq:(u-h)-diff-est} and Lemma~\ref{lem:weak-growth-est} with $\e=\al/4$, we get
 \begin{align*}
    \int_{B_1}\left(|\D(\uu_r-\h_r)|^2+|\uu_r-\h_r|\right)&\le C(n)r^\al\int_{B_1}\left(|\D\uu_r|^2+|\uu_r|\right)+C(n)\\
    &\le C(n,\al,E(\uu,1)).
\end{align*}
This, combined with Lemma~\ref{lem:quad-growth-L1}, gives \begin{align*}
    \int_{B_1}|\h_r|\le \int_{B_1}|\uu_r-\h_r|+\int_{B_1}|\uu_r|\le C(n,\al,E(\uu,1)).
\end{align*}
Since $\h_r$ is harmonic, we also have $$
\int_{B_{1/2}}|\D\h_r|^2\le C(n)\left(\int_{B_1}|\h_r|\right)^2\le C(n,\al,E(\uu,1)).
$$
Therefore, \begin{align*}
    \int_{B_{1/2}}|\D\uu_r|^2\le 2\int_{B_{1/2}}|\D\h_r|^2+2\int_{B_{1/2}}|\D(\uu_r-\h_r)|^2\le C(n,\al,E(\uu,1)).
\end{align*}
This gives the first estimate. For the second one, we observe that \begin{align*}
    &\frac2{r^{n+3}}\int_{\p B_r}|\uu|^2-\frac1{r^{n+2}}\int_{B_r}|\D\uu|^2\\
    &\qquad=\frac2{r^{n+3}}\int_{\p B_r}|\uu-\pp_{0,r}|^2-\frac1{r^{n+2}}\int_{B_r}|\D(\uu-\pp_{0,r})|^2\\
    &\qquad\le \frac2{r^{n+3}}\int_{\p B_r}|\uu-\pp_{0,r}|^2 \le C(n,\al,E(\uu,1)),
\end{align*}
where the last inequality follows from \eqref{eq:u-p-diff-est}.
This, combined with the first estimate, implies the second one.
\end{proof}

Finally, combining Lemma \ref{lem:quad-growth-L1}, Lemma \ref{lem:quad-growth-L2} and  Theorem \ref{thm:grad-u-holder} we obtain the following.

\begin{corollary}\label{UB}Let $\uu$ be an almost minimizer in $B_1$. Then, for $x_0\in \Gamma(\uu) \cap B_{1/2}$, $0<r<r_0$, $r_0=r_0(n,\alpha, E(\uu,1))$, $\uu_{x_0,r}\in C^{1,\al/2}(B_1)$ and  \begin{equation}\label{1a}
\|\uu_{x_0,r}\|_{C^{1,\al/2}(B_1)}\le C(n,\al,E(\uu,1)).
\end{equation}
\end{corollary}

\begin{remark} We remark that in Lemma 
\ref{lem:quad-growth-L1}, Lemma \ref{lem:quad-growth-L2}, and in Corollary \ref{UB}, it is possible to obtain a constant $C$ independent of $E(\uu, 1).$ This can be achieved by a standard compactness argument.\end{remark}
%%%%%%%%%%%%%%%%%%%%%%%%%%%%%%%%%%%%%%%%%%%%%%%%%%%%

\section{Non-degeneracy}
In this section we prove that almost minimizers satisfy a non-degeneracy property. Precisely, we have the following theorem.

\begin{theorem}[Non-Degeneracy]\label{ND}Let $\uu$ be an almost minimizer in $B_1$. There exist constants $c_0=c_0(n,m,\alpha, E(\uu, 1))>0$ and $r_0=r_0(n,m,\alpha)>0$ such that if $x_0\in\Gamma(\uu) \cap B_{1/2}$ and $0<r<r_0$, then $$
\sup_{B_{r}(x_0)}|\uu|\ge c_0r^2.
$$

\end{theorem}

The proof of Theorem \ref{ND} relies on the following lemma.

\begin{lemma}\label{imp}
Let $\uu$ be an almost minimizer in $B_1$. Then, there exist small constants $\e_0=\e_0(n,m)>0$ and $r_0=r_0(n,m,\al)>0$ such that for $0<r<r_0$, if $E(\uu_{x_0,r},1)\le \e_0$ then $E(\uu_{x_0,r/2},1)\le \e_0$.
\end{lemma}

\begin{proof}
For simplicity we may assume $x_0=0$. For $0<r<r_0$ to be specified later, let $\vv_r$ be a solution of $\Delta\vv_r=\frac{\vv_r}{|\vv_r|}\rchi_{\{|\vv_r|>0\}}$ in $B_1$ with $\vv_r=\uu_r$ on $\p B_1$. We claim that if $\e_0=\e_0(n,m)>0$ is small, then $\vv_r\equiv \mathbf{0}$ in $B_{1/2}$. Indeed, if not, then $\sup_{B_{3/4}}|\vv_r|\ge c_0(n)$ by the non-degeneracy of the solution $\vv_r$. Thus $|\vv_r(z_0)|\ge c_0(n)$ for some $z_0\in \overline{B_{3/4}}$. From $E(\vv_r,1)\le E(\uu_r,1)\le \e_0$ together with  the estimate for the solution $\vv_r$,
$$\sup_{B_{7/8}}|\D\vv_r|\le C_1(n,m)(E(\vv_r,1)+1)
$$ we have that $$\sup_{B_{7/8}}|\D\vv_r|\le C(n,m),$$
hence $$
|\vv_r|\ge\frac{c_0(n)}2\quad\text{in}\quad B_{\rho_0}(z_0)
$$
for some small $\rho_0=\rho_0(n,m)>0$. This gives that $$
c(n,m)\le\int_{B_{\rho_0}(z_0)}|\vv_r|\le E(\vv_r,1)\le \e_0,
$$
which is a contradiction if $\e_0=\e_0(n,m)$ is small.\\

Now, we use again that  $E(\vv_r,1)\le \e_0$ together with the fact that $\uu_r$ is an almost minimizer in $B_1$ with gauge function $\omega(\rho)=(r\rho)^\al$ to get \begin{align*}
    \int_{B_1}\left(|\D\uu_r|^2+2|\uu_r|\right)&\le (1+ r^\alpha)\int_{B_1}\left(|\D\vv_r|^2+2|\vv_r|\right)\\
    &\le \e_0 r^\al+\int_{B_1}\left(|\D\vv_r|^2+2|\vv_r|\right)
\end{align*}
for $0<r<r_0$. Then \begin{align*}
    \int_{B_1}|\D(\uu_r-\vv_r)|^2&=\int_{B_1}\left(|\D\uu_r|^2-|\D\vv_r|^2+2\D(\vv_r-\uu_r)\D\vv_r\right)\\
    &=\int_{B_1}\left(|\D\uu_r|^2-|\D\vv_r|^2\right)-2\int_{B_1}\left((\vv_r-\uu_r)\frac{\vv_r}{|\vv_r|}\chi_{\{|\vv_r|>0\}}\right)\\
    &\le \e_0 r^\al+2\int_{B_1}(|\vv_r|-|\uu_r|)-2\int_{B_1}\left(|\vv_r|-\uu_r\frac{\vv_r}{|\vv_r|}\chi_{\{|\vv_r|>0\}}\right)\\
    &\le \e_0 r^\al.
\end{align*}
Combining this with Poincar\'e's inequality and H\"older's inequality, we obtain $$
\int_{B_1}\left(|\D(\uu_r-\vv_r)|^2+2|\uu_r-\vv_r|\right)\le C(n)r^{\al/2}.
$$
Since $\vv_r\equiv\mathbf0$ in $B_{1/2}$, we see that for $0<r<r_0(n,m,\al)$, \begin{align*}
    E(\uu_r,1/2)=\int_{B_{1/2}}\left(|\D\uu_r|^2+2|\uu_r|\right)\le C(n)r^{\al/2}\le\frac{\e_0}{2^{n+2}}.
\end{align*}
Therefore, we conclude that \begin{equation*}
E(\uu_{r/2},1)=2^{n+2}E(\uu_r,1/2)\le \e_0.\qedhere
\end{equation*}
\end{proof}

Lemma \ref{imp} immediately implies the following integral form of non-degeneracy.

\begin{lemma}\label{IND}
Let $\uu$, $\e_0$ and $r_0$ be as in the preceeding lemma. If $x_0\in \overline{\{|\uu|>0\}}\cap B_{1/2}$ and $0<r<r_0$, then \begin{align}
    \label{eq:alm-min-nondeg}
    \int_{B_r(x_0)}\left(|\D\uu|^2+2|\uu|\right)\ge \e_0 r^{n+2}.
\end{align}
\end{lemma}

\begin{proof}
By the continuity of $\uu$, it is enough to prove \eqref{eq:alm-min-nondeg} for $x_0\in\{|\uu|>0\}\cap B_{1/2}$. Towards a contradiction, we suppose that $\int_{B_r(x_0)}\left(|\D\uu|^2+2|\uu|\right)\le\e_0r^{n+2}$, or equivalently $E(\uu_{x_0,r},1)\le \e_0$. Then, by the previous lemma we have $E(\uu_{x_0,r/2^k},1)\le\e_0$ for all $k\in\N$. From $|\uu(x_0)|>0$, we see that $|\uu|>c_0>0$ in $B_{r/2^k}(x_0)$ for large $k$. Therefore, \begin{align*}
    \e_0&\ge E(\uu_{x_0,r/2^k},1)=\frac1{(r/2^k)^{n+2}}\int_{B_{r/2^k}(x_0)}\left(|\D\uu|^2+2|\uu|\right)\\
    &\ge\frac1{(r/2^k)^{n+2}}\int_{B_{r/2^k}(x_0)}2c_0= \frac{C(n)c_0}{(r/2^k)^2}\to\infty\quad\text{as }k\to\infty.
\end{align*}
This is a contradiction, as desired.
\end{proof}

We are now ready to prove Theorem \ref{ND}.

\

\textit{Proof of Theorem \ref{ND}.} Assume by contradiction that 
$$\uu_{x_0,r}(x) < c_0, \quad  \text{in $ B_{1},$} $$ with $c_0$ small, to be made precise later. Let $\epsilon_0, r_0$ be the constants in Lemma \ref{imp}, and let $r < r_0$. Then, by interpolation together with estimate \eqref{1a},
\begin{align*}\|\nabla \uu_{x_0,r}\|_{L^\infty (B_{1/2})} & \leq \epsilon \|\uu_{x_0,r}\|_{C^{1,\alpha/2}(B_{3/4})} + K(\epsilon) \|\uu_{x_0,r}\|_{L^\infty(B_{3/4})}\\
&\le \epsilon  C(n,\alpha, E(\uu, 1)) + K(\epsilon)c_0 \leq \frac{\epsilon_0}{2^{n+3}},
 \end{align*}
by choosing $\epsilon= \frac{\epsilon_0}{2^{n+4}C}$ and $c_0 \leq \epsilon_0/(2^{n+4}K(\epsilon))$.  Thus, if $c_0 < \epsilon_0/2^{n+3}$, then  $E(\uu_{x_0,r}, \frac{1}{2}) < \frac{\epsilon_0}{2^{n+2}}$, which contradicts Lemma \ref{IND}.
\qed

%%%%%%%%%%%%%%%%%%%%%%%%%%%%%%%%%%%%%%%%%%%%%%

\section{$2$-Homogeneous blowups}

In this section we consider $2$-homogeneous rescalings and blowups of almost minimizers and complete the proof of Theorem~\ref{thm:Weiss}. We then obtain the polynomial decay estimate of Weiss-type energies, with the help of the epiperimetric inequality applied to solutions of \eqref{eq:sol}.

\begin{lemma}\label{lem:blowup-exist}
Suppose $\uu$ is an almost minimizer in $B_1$ and $x_0\in \Gamma(\uu)\cap B_{1/2}$. Then, for $2$-homogeneous rescalings $\uu_{x_0,t}$, there exists $\uu_{x_0,0}\in C^1_{\loc}(\R^n;\R^m)$ such that over a subsequence $t=t_j\to 0+$, $$
\uu_{x_0,t_j}\to \uu_{x_0,0}\quad\text{in }C^1_{\loc}(\R^n;\R^m).
$$
Moreover, $\uu_{x_0,0}$ is a nonzero $2$-homogeneous global solution.
\end{lemma}

\begin{proof}
For simplicity we assume $x_0=0$.

\medskip\noindent \emph{Step 1.} We first prove the $C^1$-convergence. From Corollary~\ref{UB}, there is $\uu_0\in C^1(B_{R/2};\R^m)$ such that over a subsequence $t=t_j\to 0+$,
$$
\uu_{t_j}\to\uu_0\quad\text{in }C^1(B_{R/2};\R^m).
$$
By letting $R\to\infty$ and using a Cantor's diagonal argument, we obtain that for another subsequence $t=t_j\to 0+$, 
\begin{align*}
\uu_{t_j}\to\uu_0\quad\text{in }C^1_{\loc}(\R^n;\R^m).
\end{align*}

\noindent \emph{Step 2.}
It follows from the non-degeneracy, $\sup_{B_1}|\uu_{t_j}|\ge c_0>0$, and the $C^1$-convergence of $\uu_{t_j}$ to $\uu_0$ that $\uu_0$ is nonzero. To show that $\uu_0$ is a global solution, for fixed $R>1$ and small $t_j$, let $\vv_{t_j}$ be the solution in $B_R$ with $\vv_{t_j}=\uu_{t_j}$ on $\p B_R$. Then, by elliptic theory, \begin{align*}
    \|\vv_{t_j}\|_{C^{1,\al/2}(\overline{B_R};\R^m)}\le C(n,m,R)(\|\uu_{t_j}\|_{C^{1,\al/2}(\overline{B_R};\R^m)}+1)\le C(n,m,\al,R,E(\uu,1)).
\end{align*}
Thus, there exists a solution $\vv_0\in C^1(\overline{B_R};\R^m)$ such that $$
\vv_{t_j}\to \vv_0\quad\text{in }C^1(\overline{B_R};\R^m).
$$
Moreover, we use again that $\uu_{t_j}$ is an almost minimizer in $B_{1/2t_j}$ with a gauge function $\omega(\rho)=(t_j\rho)^\al$ to have $$
\int_{B_R}\left(|\D\uu_{t_j}|^2+2|\uu_{t_j}|\right)\le (1+(t_jR)^\al)\int_{B_R}\left(|\D\vv_{t_j}|^2+2|\vv_{t_j}|\right).
$$
By taking $t_j\to 0$ and using the $C^1$-convergence of $\uu_{t_j}$ and $\vv_{t_j}$, we obtain 
$$
\int_{B_R}\left(|\D\uu_0|^2+2|\uu_0|\right)\le \int_{B_R}\left(|\D\vv_0|^2+2|\vv_0|\right).
$$
Since $\vv_{t_j}=\uu_{t_j}$ on $\p B_R$, we also have $\vv_0=\uu_0$ on $\p B_R$. This means that $\uu_0=\vv_0$ is an energy minimizer, or a solution in $B_R$. Since $R>1$ is arbitrary, we conclude that $\uu_0$ is a global solution.

\medskip\noindent \emph{Step 3.} Now, we prove that $\uu_0$ is $2$-homogeneous. Fix $0<r<R<\infty$. By Theorem~\ref{thm:weiss-mon}, we have that for small $t_j$, \begin{align}\begin{split}\label{eq:W-diff-bound}
    &W(\uu,Rt_j)-W(\uu,rt_j)\\
    &\qquad=\int_{rt_j}^{Rt_j}\frac{d}{d\rho}W(\uu,\rho)\,d\rho\\
    &\qquad\ge \int_{rt_j}^{Rt_j}\frac1{\rho^{n+4}}\int_{\p B_\rho}|x\cdot\D\uu-2(1-b\rho^\al)\uu|^2\,dS_xd\rho\\
    &\qquad=\int_r^R\frac{t_j}{(t_j\si)^{n+4}}\int_{\p B_{t_j\si}}|x\cdot\D\uu-2(1-b(t_j\si)^\al)\uu|^2\,dS_xd\si\\
    &\qquad=\int_r^R\frac1{t_j^4\si^{n+4}}\int_{\p B_\si}|t_jx\cdot\D\uu(t_jx)-2(1-b(t_j\si)^\al)\uu(t_jx)|^2\,dS_xd\si\\
    &\qquad=\int_r^R\frac1{\si^{n+4}}\int_{\p B_\si}|x\cdot\D\uu_{t_j}-2(1-b(t_j\si)^\al)\uu_{t_j}|^2\,dS_xd\si.
\end{split}\end{align}
Note that in \eqref{eq:L1-bound} we showed $$
\frac{2e^{ar^\al}}{r^{n+2}}\int_{B_r}|\uu|\le W(\uu,r)+\frac{2e^{ar^\al}}{r^{n+3}}\int_{\p B_r}|\uu-\pp_{0,r}|^2.
$$
Since we also proved in the proof of Lemma~\ref{lem:quad-growth-L1} that $$
\frac{2e^{ar^\al}}{r^{n+3}}\int_{\p B_r}|\uu-\pp_{0,r}|^2\le C_0,
$$
we see that $W(\uu,0+)\ge -C_0$ and thus $W(\uu,0+)\in (-\infty,\infty)$. Using this and taking $t_j\to 0+$ in \eqref{eq:W-diff-bound}, we get \begin{align*}
    0=W(\uu,0+)-W(\uu,0+)\ge \int_r^R\frac1{\si^{n+3}}\int_{\p B_\si}|x\cdot\D\uu_0-2\uu_0|^2\,dS_xd\si.
\end{align*}
Taking $r\to0+$ and $R\to\infty$, we conclude that $x\cdot\D\uu_0-2\uu_0=0$ in $\R^n$, which implies that $\uu_0$ is $2$-homogeneous in $\R^n$.
\end{proof}

\begin{proof}[Proof of Theorem~\ref{thm:Weiss}]
The monotonicity of the Weiss energy $W(\uu,x_0,\cdot)$ follows from Theorem~\ref{thm:weiss-mon}. To prove \eqref{eq:Weiss-limit-classify}, we observe that for $\uu_{x_0,t}=\frac{\uu(tx+x_0)}{t^2}$, \begin{align*}
    W(\uu,x_0,t)&=e^{at^\al}\left(\frac1{t^{n+2}}\int_{B_t(x_0)}\left(|\D\uu|^2+2|\uu|\right)-\frac{2(1-bt^\al)}{t^{n+3}}\int_{\p B_t(x_0)}|\uu|^2\right)\\
    &=e^{at^\al}\left(\int_{B_1}\left(|\D\uu_{x_0,t}|^2+2|\uu_{x_0,t}|\right)-2(1-bt^\al)\int_{\p B_1}|\uu_{x_0,t}|^2\right).
\end{align*}
By Lemma~\ref{lem:blowup-exist}, we can take a subsequence $t=t_j\to 0+$ such that $\uu_{x_0,t_j}\to \uu_{x_0,0}$ in $C^1_{\loc}(\R^n;\R^m)$ for some nonzero $2$-homogeneous solution $\uu_{x_0,0}$, which gives \begin{align*}
    W(\uu,x_0,0+)=M(\uu_{x_0,0}).
\end{align*}
Therefore, by Corollary 1 in \cite{AndShaUraWei15}, \begin{equation*}
    W(\uu,x_0,0+)= \be_n/2\quad\text{or}\quad W(\uu,x_0,0+)\ge \bar{\be}_n.\qedhere
\end{equation*}
\end{proof}

The following is the epiperimetric inequality from \cite{AndShaUraWei15}.
\begin{theorem}\label{epi}
There exists $\kappa\in (0,1)$ and $\delta>0$ such that if
$\mathbf{c}$ is a homogeneous
function of degree $2$ satisfying $\| \mathbf{c}-\h\|_{W^{1,2}(B_1;\R^m)}+\| \mathbf{c}-\h\|_{L^\infty(B_1;\R^m)}\le \delta$ for some $\h \in \Hf$,
then there is $\vv\in W^{1,2}(B_1;\R^m)$ such that $\vv=\mathbf{c}$ on $\partial B_1$ and
\begin{equation*}
M(\vv) \le (1-\kappa) M(\mathbf{c}) + \kappa \frac{\be_n}{2}.
\end{equation*}
\end{theorem}

\begin{lemma}
\label{lem:Weiss-decay}
Let $\uu$ be an almost minimizer in $B_1$ and $x_0\in \mathcal{R}_\uu\cap B_{1/2}$. Suppose that the epiperimetric inequality holds with $\kappa\in(0,1)$ for every homogeous replacement $\cc_{x_0,r}$. Then $$
W(\uu,x_0,r)-W(\uu,x_0,0+)\le C(n,\al,E(\uu,1))r^\de,\quad 0<r<r_0=r_0(n,\al)
$$
for some $\de=\de(n,\al,\kappa)>0$.
\end{lemma}

\begin{proof}
For simplicity, we assume $x_0=0$, and write $W(r)=W(\uu,0,r)$ and $W(0+)=W(\uu,0,0+)$. We define \begin{align*}
    e(r)&:=W(r)-W(0+)\\
    &=\frac{e^{ar^\al}}{r^{n+2}}\int_{B_r}\left(|\D\uu|^2+2|\uu|\right)-\frac{2(1-br^\al)e^{ar^\al}}{r^{n+3}}\int_{\p B_r}|\uu|^2-W(0+).
\end{align*} 
Then, from the identities $$
\frac d{dr}\left(\frac{e^{ar^\al}}{r^{n+2}}\right)=-\frac{(n+2)(1-r^\al)}r\cdot\frac{e^{ar^\al}}{r^{n+2}}$$ and $$\frac d{dr}\left(\frac{2(1-br^\al)e^{ar^\al}}{r^{n+3}}\right)=\frac{-2(1-br^\al)e^{ar^\al}(n+3+O(r^\al))}{r^{n+4}},
$$
we have \begin{align*}
    e'(r)&=-\frac{(n+2)(1-r^\al)}{r}\frac{e^{ar^\al}}{r^{n+2}}\int_{B_r}\left(|\D\uu|^2+2|\uu|\right)+\frac{e^{ar^\al}}{r^{n+2}}\int_{\p B_r}\left(|\D\uu|^2+2|\uu|\right)\\
    &\qquad+\frac{2(1-br^\al)e^{ar^\al}(n+3+O(r^\al))}{r^{n+4}}\int_{\p B_r}|\uu|^2\\
    &\qquad-\frac{2(1-br^\al)e^{ar^\al}}{r^{n+3}}\left(\int_{\p B_r}2\uu\p_\nu\uu+\frac{n-1}r\int_{\p B_r}|\uu|^2\right)\\
    &\ge -\frac{n+2}r\left(e(r)+\frac{2(1-br^\al)e^{ar^\al}}{r^{n+3}}\int_{\p B_r}|\uu|^2+W(0+)\right)\\
    &\qquad\begin{multlined}+\frac{e^{ar^\al}}r\bigg[\frac{1}{r^{n+1}}\int_{\p B_r}\left(|\D\uu|^2+2|\uu|\right)+\frac{8(1-br^\alpha)+O(r^\al)}{r^{n+3}}\int_{\p B_r}|\uu|^2\\
     -\frac{2(1-br^\al)}{r^{n+2}}\int_{\p B_r}2\uu\p_\nu\uu\bigg]\end{multlined}\\
    &\ge -\frac{n+2}r\left(e(r)+W(0+)\right)\\
    &\qquad\begin{multlined}+\frac{(1-br^\al)e^{ar^\al}}{r}\bigg[\frac1{r^{n+1}}\int_{\p B_r}\left(|\D\uu|^2+2|\uu|\right)\\
    +\frac{-2n+4+O(r^\al)}{r^{n+3}}\int_{\p B_r}|\uu|^2-\frac4{r^{n+2}}\int_{\p B_r}\uu\p_\nu\uu\bigg].\end{multlined}
\end{align*}
Combining this with the computation \begin{align*}
    &\int_{\p B_r}\left(\frac1{r^{n+1}}\left(|\D\uu|^2+2|\uu|\right)+\frac{-2n+4+O(r^\al)}{r^{n+3}}|\uu|^2-\frac4{r^{n+2}}\uu\p_\nu\uu\right)\\
    &=\int_{\p B_1}\left(|\D\uu_r|^2+2|\uu_r|+(-2n+4+O(r^\al))|\uu_r|^2-4\uu_r\p_\nu\uu_r\right)\\
    &=\int_{\p B_1}\left(|\p_\nu\uu_r-2\uu_r|^2+|\p_\theta\uu_r|^2+2|\uu_r|-(2n+O(r^\al))|\uu_r|^2\right)\\
    &\ge \int_{\p B_1}\left(|\p_\theta\cc_r|^2+2|\cc_r|-(2n+O(r^\al))|\cc_r|^2\right)\\
    &=\int_{\p B_1}\left(|\D\cc_r|^2+2|\cc_r|-(2n+4+O(r^\al))|\cc_r|^2\right)\\
    &=(n+2)\left[\int_{B_1}\left(|\D\cc_r|^2+2|\cc_r|\right)-(2+O(r^\al))\int_{\p B_1}|\cc_r|^2\right]\\
    &=(n+2)M(\cc_r)+O(r^\al)\int_{\p B_1}|\uu_r|^2,
\end{align*}
we get $$
e'(r)\ge-\frac{n+2}r(e(r)+W(0+))+\frac{(1-br^\al)e^{ar^\al}(n+2)}rM(\cc_r)+\frac{O(r^\al)}r\int_{\p B_1}|\uu_r|^2.
$$
To estimate $M(\cc_r)$, we use the almost minimizing property of $\uu$ and that the epiperimetric inequality $M(\vv)\le (1-\kappa)M(\cc_r)+\kappa\frac{\be_n}2$ holds for some $\vv\in W^{1,2}(B_1;\R^m)$ with $\vv=\cc_r=\uu_r$ on $\p B_1$ to obtain \begin{align*}
    (1-\kappa)M(\cc_r)+\kappa W(0+)&\ge M(\vv)\\
    &=\int_{B_1}\left(|\D\vv|^2+2|\vv|\right)-2\int_{\p B_1}|\vv|^2\\
    &\ge \frac1{1+r^\al}\int_{B_1}\left(|\D\uu_r|^2+2|\uu_r|\right)-2\int_{\p B_1}|\uu_r|^2\\
    &=\frac{e^{-ar^\al}}{1+r^\al}W(r)+O(r^\al)\int_{\p B_1}|\uu_r|^2.
\end{align*}
This gives $$
M(\cc_r)\ge \frac{\frac{e^{-ar^\al}}{1+r^\al}W(r)-\kappa W(0+)}{1-\kappa}+O(r^\al)\int_{\p B_1}|\uu_r|^2.
$$
Thus \begin{align*}
    e'(r)&\ge -\frac{n+2}r(e(r)+W(0+))+\frac{(1-br^\al)e^{ar^\al}(n+2)}r\left(\frac{\frac{e^{-ar^\al}}{1+r^\al}W(r)-\kappa W(0+)}{1-\kappa}\right)\\
    &\qquad+\frac{O(r^\al)}r\int_{\p B_1}|\uu_r|^2.
\end{align*}
Note that from Lemmas~\ref{lem:quad-growth-L1} and ~\ref{lem:quad-growth-L2}, we have $$
W(0+)\le W(r)\le C(n,\al,E(\uu,1))
$$
and $$
\int_{\p B_1}|\uu_r|^2=\frac1{r^{n+3}}\int_{\p B_r}|\uu|^2\le C(n,\al,E(\uu,1))
$$
for $0<r<r_0(n,\al)$. Using these estimates, we obtain \begin{align*}
    e'(r)&\ge -\frac{n+2}r(e(r)+W(0+))+\frac{n+2}r\left(\frac{W(r)-\kappa W(0+)}{1-\kappa}\right)\\
    &\qquad+\frac{O(r^\al)}r\left(W(r)+W(0+)+\int_{\p B_1}|\uu_r|^2\right)\\
    &\ge -\frac{n+2}re(r)+\frac{n+2}r\left(\frac{W(r)-W(0+)}{1-\kappa}\right)-C_0r^{\al-1}\\
    &=\left(\frac{(n+2)\kappa}{1-\kappa}\right)\frac{e(r)}r-C_0r^{\al-1},\quad 0<r<r_0(n,\al),
\end{align*}
for some $C_0=C_0(n,\al,E(\uu,1))$. Now, we take $\de=\de(n,\al,\kappa)$ such that $0<\de<\min\{\frac{(n+2)\kappa}{1-\kappa},\al\}$. Then, using the above differential inequality for $e(r)$ and that $e(r)=W(r)-W(0+)\ge 0$, we obtain \begin{align*}
    \frac d{dr}\left[e(r)r^{-\de}+\frac{C_0}{\al-\de}r^{\al-\de}\right]
    &=r^{-\de}\left[e'(r)-\frac\de re(r)\right]+C_0r^{\al-\de-1}\\
    &\ge r^{-\de}\left[\left(\frac{(n+2)\kappa}{1-\kappa}-\de\right)\frac{e(r)}r-C_0r^{\al-1}\right]+C_0r^{\al-\de-1}\\
    &\ge 0.
\end{align*}
Thus $$
e(r)r^{-\de}\le e(r)r^{-\de}+\frac{C_0}{\al-\de}r^{\al-\de}\le e(r_0)r_0^{-\de}+\frac{C_0}{\al-\de}r_0^{\al-\de},
$$
and hence we conclude that \begin{equation*}
W(r)-W(0+)=e(r)\le C(n,\al,E(\uu,1))r^\de.\qedhere
\end{equation*}
\end{proof}

Now, we consider an auxiliary
function
$$
\phi(r):=e^{-(2 b/\alpha)r^\alpha}r^{2},\quad
r>0,
$$
which is a solution of the differential equation
$$
\phi'(r)=2\,\phi(r)\frac{1-b r^\alpha}r,\quad r>0.
$$
For $x_0\in B_{1/2}$, we define the \emph{$2$-almost homogeneous rescalings} by
$$
\uu_{x_0, r}^{\phi}(x):=\frac{\uu(rx+x_0)}{\phi(r)},\quad x\in B_{1/(2r)}.
$$

\begin{lemma}
[Rotation estimate]\label{lem:rot-est}
Under the same assumptions as in Lemma~\ref{lem:Weiss-decay},
$$
\int_{\p B_1}|\uu_{x_0,t}^\phi-\uu_{x_0,s}^\phi|\le C(n,\al,\kappa,E(\uu,1))t^{\de/2},\quad s<t<t_0(n,\al).
$$
\end{lemma}

\begin{proof}
Without loss of generality, assume $x_0=0$. Then, for $\uu_r^\phi=\uu_{0,r}^\phi$,
  \begin{align*}
    \frac{d}{dr}\uu_{r}^{\phi}(x)
    &=\frac{\nabla \uu(rx)\cdot x}{\phi(r)}-\frac{\uu(rx)[\phi'(r)/\phi(r)]}{\phi(r)}\\
    &=\frac{1}{\phi(r)}\left(\nabla \uu(rx)\cdot
      x-\frac{2(1-br^\alpha)}{r}\uu(rx)\right).
  \end{align*}
Thus, for $0<r<t_0=t_0(n,\alpha)$, using Theorem \ref{thm:weiss-mon},
\begin{align*}
  &\left(\int_{\partial
    B_1}\left[\frac{d}{dr}\uu_r^\phi(\xi)\right]^2dS_\xi\right)^{1/2}\\
  &=
    \frac{1}{\phi(r)}\left(\int_{\partial B_1}\left|\nabla \uu(r\xi)\cdot \xi
    -\frac{2(1-br^\alpha)}{r}\uu(r \xi)\right|^2  dS_\xi\right)^{1/2}\\
  &=\frac{1}{\phi(r)}\left(\frac{1}{r^{n-1}}\int_{\partial
    B_r}\left|\D \uu(x)\cdot \nu-\frac{2(1-br^\alpha)}{r}
    \uu(x)\right|^2dS_x\right)^{1/2}\\
  &\leq\frac{1}{\phi(r)}\left(\frac{1}{r^{n-1}}\frac{r^{n+2}}{e^{a
    r^\alpha}}\frac{d}{dr}W(\uu,r)\right)^{1/2}=\frac{e^{c
    r^\alpha}}{r^{1/2}}\left(\frac{d}{dr}W(\uu,r)\right)^{1/2}
    ,\quad c=\frac{2b}{\alpha}-\frac{a}{2}.
\end{align*}
Using this and Lemma~\ref{lem:Weiss-decay}, we can compute \begin{align*}
    \int_{\p B_1}|\uu_{t}^\phi-\uu_{s}^\phi|&\le \int_{\p B_1}\int_s^t\left|\frac{d}{dr}\uu_{r}^\phi\right|\,dr=\int_s^t\int_{\p B_1}\left|\frac{d}{dr}\uu_{r}^\phi\right|\,dr\\
    &\le C_n\int_s^t\left(\int_{\p B_1}\left|\frac{d}{dr}\uu_{r}^\phi\right|^2\right)^{1/2}\,dr\\
    &\le C_n\left(\int_s^t\frac1r\,dr\right)^{1/2}\left(\int_s^tr\int_{\p B_1}\left|\frac{d}{dr}\uu_{r}^\phi\right|^2\,dr\right)^{1/2}\\
    &\le C_ne^{ct^\al}\left(\log\frac ts\right)^{1/2}\left(\int_s^t\frac{d}{dr}W(\uu,r)\,dr\right)^{1/2}\\
    &\le C\left(\log\frac ts\right)^{1/2}(W(\uu,t)-W(\uu,s))^{1/2}\\
    &\le C\left(\log\frac ts\right)^{1/2}t^{\de/2},\quad 0<t<t_0(n,\al).
\end{align*}
Now, by a standard dyadic argument, we conclude that \begin{equation*}
\int_{\p B_1}|\uu_{t}^\phi-\uu_{s}^\phi|\le Ct^{\de/2}.\qedhere
\end{equation*}
\end{proof}

%%%%%%%%%%%%%%%%%%%%%%%%%%%%%%%%%%%%%%%%%%%%%%%%%%

\section{Regularity of the regular set}
In this last section we prove one of the most important results in this paper, the $C^{1,\gamma}$ regularity of the regular set.

\begin{lemma}\label{lem:rescaling-H-dist}
Let $C_h$ be a compact subset of $\mathcal{R}_\uu$.  Then for small $r$, the $2$-homogeneous replacements $\cc_{x_0,r}$ are uniformly (in $x_0\in C_h$) close to $\Hf$ in the $C^1(\overline{B_1};\R^m)$-topology.
\end{lemma}

\begin{proof}
We first claim that $\uu_{x_0,r}$ is uniformly close to $\Hf$ in $C^1(\overline{B_1};\R^m)$. Indeed, assume towards a contradiction that there exist $x_j\in C_h$, $r_j\to 0$ and $\e_0>0$ such that \begin{align}
    \label{eq:rescaling-H-dist}
    \dist(\uu_{x_j,r_j},\Hf)\ge \e_0,
\end{align}
where the distance is measured in the $C^1(\overline{B_1};\R^m)$-norm. Following the argument in the proof of Lemma~\ref{lem:blowup-exist} with $\uu_{x_j,r_j}$ in the place of $\uu_{t_j}$, we obtain that up to a subsequence, $$
\uu_{x_j,r_j}\to\w\quad\text{in }C^1_{\loc}(\R^n;\R^m)
$$
for some nonzero $2$-homogeneous global solution $\w\in C^1_{\loc}(\R^n;\R^m)$. Then, by Corollary 1 in \cite{AndShaUraWei15}, we have either $M(\w)=\be_n/2$ or $M(\w)=\bar{\be}_n$. Moreover, by Dini's theorem, there is $r_0>0$ such that $$
W(\uu,x_j,r)<1/2(\be_n/2+\bar{\be}_n)
$$
for all $0<r<r_0$ and all $x_j\in C_h$. Thus $$
W(\w,0,1)=\lim_{j\to\infty}W(\uu_{x_j,r_j},0,1)=\lim_{j\to\infty}W(\uu,x_j,r_j)=\be_n/2,
$$
and hence $\w\in\Hf$ by Corollary 1 in \cite{AndShaUraWei15}. This contradicts \eqref{eq:rescaling-H-dist}, and we conclude that $\uu_{x_0,r}$ is uniformly close to $\Hf$ in $C^1(\overline{B_1};\R^m)$.\\
Now, to show that $\cc_{x_0,r}$ is close to $\Hf$, let $\e>0$ be given. Then there is $r_\e>0$ such that if $x_0\in C_h$ and $0<r<r_\e$, then $\|\uu_{x_0,r}-\h_{x_0,r}\|_{C^1(\overline{B_1};\R^m)}<\e$ for some $\h_{x_0,r}\in\Hf$. Since $\cc_{x_0,r}=|x|^2\uu_{x_0,r}\left(\frac{x}{|x|}\right)$ is $2$-homogeneous $C^{1,\al/2}$-regular, we have for some universal constant $C>0$ \begin{align*}
  &\|\cc_{x_0,r}-\h_{x_0,r}\|_{C^1(\overline{B_1};\R^m)}\\
  &\qquad\le C\left(\|\cc_{x_0,r}-\h_{x_0,r}\|_{L^\infty(\p B_1;\R^m)}+\|\D_\theta(\cc_{x_0,r}-\h_{x_0,r})\|_{L^\infty(\p B_1;\R^m)}\right). 
\end{align*}
Moreover, from the fact that $\cc_{x_0,r}$ coincides with $\uu_{x_0,r}$ on $\p B_1$, we conclude that \begin{equation*}
\|\cc_{x_0,r}-\h_{x_0,r}\|_{C^1(\overline{B_1};\R^m)}=C\|\uu_{x_0,r}-\h_{x_0,r}\|_{C^1(\overline{B_1};\R^m)}<C\e,
\end{equation*}
as desired.
\end{proof}

\begin{lemma}\label{lem:rescaling-blowup-est}
Let $\uu$ be an almost minimizer in $B_1$, $C_h$ a compact subset of $\mathcal{R}_\uu$, and $\de$ as in Lemma~\ref{lem:Weiss-decay}. Then for every $x_0\in C_h$ there is a unique blowup $\uu_{x_0,0}\in\Hf$. Moreover, there exists $r_0>0$ and $C>0$ such that $$
\int_{\p B_1}|\uu^\phi_{x_0,r}-\uu_{x_0,0}|\le Cr^{\de/2}
$$
for all $0<r<r_0$ and $x_0\in C_h$.
\end{lemma}

\begin{proof}
By the definition of $\mathcal{R}_\uu$ and Corollary 1 in \cite{AndShaUraWei15}, every blowup $\uu_{x_0,0}$ at $x_0$ is in $\Hf$. Moreover, by Lemma~\ref{lem:rot-est} and Lemma~\ref{lem:rescaling-H-dist}, $$
\int_{\p B_1}|\uu_{x_0,r}^\phi-\uu_{x_0,s}^\phi|\le Cr^{\de/2},\quad s<r<r_0.
$$
Note that $\uu_{x_0,t}^\phi(x)=\frac{\uu(tx+x_0)}{\phi(t)}$ and $\lim_{t\to 0}\frac{\phi(t)}{t^2}=1$. Thus, if $\uu_{x_0,0}$ is the limit of $\uu_{x_0,t_j}$ for $t_j\to 0$, then it is also the limit of $\uu_{x_0,t_j}^\phi$. Taking $s=t_j\to 0$ in the above estimate and passing to the limit, we get $$
\int_{\p B_1}|\uu^\phi_{x_0,r}-\uu_{x_0,0}|\le Cr^{\de/2}.
$$
Finally, to prove the uniqueness of blowup, let $\tilde{\uu}_{x_0,0}$ be another blowup. Then,$$
\int_{\p B_1}|\tilde{\uu}_{x_0,0}-\uu_{x_0,0}|=0.
$$
By the homogeneity, we conclude that $\uu_{x_0,0}=\tilde{\uu}_{x_0,0}$ in $\R^n$.
\end{proof}

Now we are ready to prove the main result on the regularity of the regular set.

\begin{proof}[Proof of Theorem~\ref{thm:reg-set}]
The relative openness of the regular set immediately follows from the fact that $x\longmapsto W(\uu,x,0+)$ is upper-semicontinous and \eqref{eq:Weiss-limit-classify}. For the regularity of $\mathcal{R}_\uu$, we follow the argument in Theorem~5 in \cite{AndShaUraWei15}.

\medskip\noindent\emph{Step 1.} Let $x_0\in\mathcal{R}_\uu$. By Lemma~\ref{lem:rescaling-blowup-est}, there exists $\rho_0>0$ such that $B_{2\rho_0}(x_0)\subset B_1$, $B_{2\rho_0}(x_0)\cap\Gamma(\uu)=B_{2\rho_0}(x_0)\cap\mathcal{R}_\uu$ and $$
\int_{\p B_1}|\uu_{x_1,r}^\phi-\frac12\mathbf{e}(x_1)\max(x\cdot\nu(x_1),0)^2|\le Cr^{\de/2},
$$
for any $x_1\in \Gamma(\uu)\cap \overline{B_{\rho_0}(x_0)}$ and for any $0<r<\rho_0$. We then claim that $x_1\longmapsto\nu(x_1)$ and $x_1\longmapsto\mathbf{e}(x_1)$ are H\"older continuous of order $\gamma$ on $\Gamma(\uu)\cap\overline{B_{\rho_1}(x_0)}$ for some $\gamma=\gamma(n,\al,\kappa)>0$ and $\rho_1\in(0,\rho_0)$. Indeed, we observe that for $x_1$ and $x_2$ near $x_0$ and for small $r>0$,
\begin{align*}
    &\frac12\int_{\p B_1}|\mathbf{e}(x_1)\max(x\cdot\nu(x_1),0)^2-\mathbf{e}(x_2)\max(x\cdot\nu(x_2),0)^2|\,dS_x\\
    &\qquad\le 2Cr^{\de/2}+\int_{\p B_1}|\uu_{x_1,r}^\phi-\uu_{x_2,r}^\phi|\\
    &\qquad\le 2Cr^{\de/2}+\frac1{\phi(r)}\int_{\p B_1}\int_0^1|\D\uu(rx+(1-t)x_1+tx_2)||x_1-x_2|\,dt\,dS_x\\
    &\qquad\le 2Cr^{\de/2}+C\frac{|x_1-x_2|}{r^2}.
\end{align*}
Now, if we take $\rho_1$ small and $x_1$, $x_2\in \overline{B_{\rho_1}(x_0)}\cap\Gamma(\uu)$ and choose $r=|x_1-x_2|^{\frac2{4+\de}}$, then $$
\frac12\int_{\p B_1}|\mathbf{e}(x_1)\max(x\cdot\nu(x_1),0)^2-\mathbf{e}(x_2)\max(x\cdot\nu(x_2),0)^2|\le C|x_1-x_2|^\gamma,\quad\gamma=\frac\de{4+\de}.
$$
Moreover, we also have (see the proof of Theorem~5 in \cite{AndShaUraWei15}) \begin{align*}
\begin{multlined}\frac12\int_{\p B_1}|\mathbf{e}(x_1)\max(x\cdot\nu(x_1),0)^2-\mathbf{e}(x_2)\max(x\cdot\nu(x_2),0)^2\\
\ge c(n)(|\nu(x_1)-\nu(x_2)|+|\mathbf{e}(x_1)-\mathbf{e}(x_2)|),
\end{multlined}\end{align*}
which readily implies the H\"older continuity of $x_1\longmapsto\nu(x_1)$ and $x_1\longmapsto\mathbf{e}(x_1)$.

\medskip\noindent\emph{Step 2.} We claim that for every $\e\in(0,1)$, there exists $\rho_\e\in(0,\rho_1)$ such that for $x_1\in\Gamma(\uu)\cap\overline{B_{\rho_1}(x_0)}$ and $y\in\overline{B_{\rho_\e}(x_1)}$, \begin{align}
    &\label{eq:reg-set-cone-1}\uu(y)=0\quad\text{if } (y-x_1)\cdot\nu(x_1)<-\e|y-x_1|,\\
    &\label{eq:reg-set-cone-2}|\uu(y)|>0\quad\text{if }(y-x_1)\cdot\nu(x_1)>\e|y-x_1|.
\end{align}
Indeed, if \eqref{eq:reg-set-cone-1} does not hold, then we can take a sequence $\Gamma(\uu)\cap\overline{B_{\rho_1}(x_0)}\ni x_j\to\bar{x}$ and a sequence $y_j-x_j\to 0$ as $j\to\infty$ such that $$
|\uu(y_j)|>0\quad\text{and}\quad(y_j-x_j)\cdot\nu(x_j)<-\e|y_j-x_j|.
$$
Then we consider $\uu_j(x):=\frac{\uu(x_j+|y_j-x_j|x)}{|y_j-x_j|^2}$ and observe that for $z_j:=\frac{y_j-x_j}{|y_j-x_j|}\in\p B_1$, $|\uu_j(z_j)|>0$ and $z_j\cdot\nu(x_j)<-\e|z_j|$. Moreover, over a subsequence, $\uu_j\to\uu_0=\frac12\mathbf{e}(\bar{x})\max(x\cdot\nu(\bar{x}),0)^2$ in $C^1_{\loc}(\R^n;\R^m)$ (see the proof of Lemma~\ref{lem:rescaling-H-dist}). Then, for $K:=\{z\in\p B_1\,:\,z\cdot\nu(\bar{x})\le-\e/2|z|\}$, we have that $z_j\in K$ for large $j$ by Step 1. We also consider a bigger compact set $\tilde{K}:=\{z\in\R^n\,:\,1/2\le|z|\le2, z\cdot\nu(\bar{x})\le-\e/4|z|\}$, and let $t:=\min\{\dist(K,\p\tilde{K}),r_0\}$, where $r_0=r_0(n,m,\al)$ is as in Lemma~\ref{imp}, so that $B_t(z_j)\subset\tilde{K}$. By applying Lemma~\ref{IND}, we obtain $$
\sup_{\tilde{K}}\left(|\D\uu_j|^2+2|\uu_j|\right)\ge C(n,t)\int_{B_t(z_j)}\left(|\D\uu_j|^2+2|\uu_j|\right)\ge c(n,m,\al,\e),
$$
which gives $$
\sup_{\tilde{K}}\left(|\D\uu_0|+|\uu_0|\right)>0.
$$
However, this is a contradiction since $\uu_0(x)=\frac12\mathbf{e}(\bar{x})\max(x\cdot\nu(\bar{x}),0)^2=0$ in $\tilde{K}$.\\
On the other hand, if \eqref{eq:reg-set-cone-2} is not true, then we take a sequence $\Gamma(\uu)\cap\overline{B_{\rho_1}(x_0)}\ni x_j\to\bar{x}$ and a sequence $y_j-x_j\to 0$ such that $|\uu(y_j)|=0$ and $(y_j-x_j)\cdot\nu(x_j)>\e|y_j-x_j|$. For $\uu_j$, $\uu_0(x)=\frac12\mathbf{e}(\bar{x})\max(x\cdot\nu(\bar{x}),0)^2$ and $z_j$ as above, we will have that $\uu_j(z_j)=0$ and $z_j\in K':=\{z\in\p B_1\,:\,z\cdot\nu(\bar{x})\ge\e/2|z|\}$. We may assume $z_j\to z_0\in K'$, and obtain $\uu_0(z_0)=0$, which is a contradiction.

\medskip\noindent\emph{Step 3.} By rotations we may assume that $\nu(x_0)=\mathbf{e}_n$ and $\mathbf{e}(x_0)=\mathbf{e}_1$. Fixing $\e=\e_0$, by Step 2 and the standard arguments, we conclude that there exists a Lipschitz function $g:\R^{n-1}\to\R$ such that for some $\rho_{\e_0}>0$,\begin{align*}
    &B_{\rho_{\e_0}}(x_0)\cap\{\uu=0\}=B_{\rho_{\e_0}}(x_0)\cap\{x_n\le g(x')\},\\
    &B_{\rho_{\e_0}}(x_0)\cap\{|\uu|>0\}=B_{\rho_{\e_0}}(x_0)\cap\{x_n> g(x')\}.
\end{align*}
Now, taking $\e\to0$, we can see that $\Gamma(\uu)$ is differentiable at $x_0$ with normal $\nu(x_0)$. Recentering at any $\bar{x}\in B_{\rho_{\e_0}}(x_0)\cap\Gamma(\uu)$ and using the H\"older continuity of $\bar{x}\longmapsto\nu(\bar{x})$, we conclude that $g$ is $C^{1,\gamma}$. This completes the proof.
\end{proof}

%%%%%%%%%%%%%%%%%%%%%%%%%%%%%%%%%%%%%%%%%%%%%%%%%%

\appendix
\section{Example of almost minimizers}\label{sec:ex}

\begin{example}\label{ex:alm-min-transport}
Let $\uu$ be a solution of the system $$
\Delta \uu+\mathbf{b}(x)\D\uu-\frac\uu{|\uu|}\rchi_{\{|\uu|>0\}}=\mathbf0\quad\text{in }B_1,
$$
where $\mathbf{b}\in L^p(B_1;\R^m)$, $p>n$, is the velocity field. Then $\uu$ is an almost minimizer with a gauge function $\omega(r)=C(n,p,\|\mathbf{b}\|_{L^p(B_1;\R^m)})r^{1-n/p}$.
\end{example}

\begin{proof}
Let $B_r(x_0)\Subset B_1$ and $\vv\in \uu+W^{1,2}_0(B_r(x_0))$. Then $\uu$ satisfies the weak equation $$
\int_{B_r(x_0)}\left(\D\uu\cdot\D(\uu-\vv)-(\uu-\vv)\mathbf{b}\D\uu+(\uu-\vv)\frac{\uu}{|\uu|}\chi_{\{|\uu|>0\}}\right)=0.
$$
Thus \begin{align*}
    \int_{B_r(x_0)}|\D\uu|^2&=\int_{B_r(x_0)}\left(\D\uu\cdot\D\vv+(\uu-\vv)\mathbf{b}\D\uu+(\vv-\uu)\frac\uu{|\uu|}\chi_{\{|\uu|>0\}}\right)\\
    &\le\int_{B_r(x_0)}\left(1/2|\D\uu|^2+1/2|\D\vv|^2+(\uu-\vv)\mathbf{b}\D\uu+|\vv|-|\uu|\right).
\end{align*}
It follows that \begin{align}
    \label{eq:alm-min-ex}
    \int_{B_r(x_0)}\left(|\D\uu|^2+2|\uu|\right)\le\int_{B_r(x_0)}\left(|\D\vv|^2+2|\vv|+2(\uu-\vv)\mathbf{b}\D\uu\right).
\end{align}
To estimate the last term, we apply H\"older's inequality to have $$
\int_{B_r(x_0)}2(\uu-\vv)\mathbf{b}\D\uu\le 2\|\uu-\vv\|_{L^{p^*}(B_r(x_0);\R^m)}\|\mathbf{b}\|_{L^p(B_r(x_0);\R^m)}\|\D\uu\|_{L^2(B_r(x_0);\R^m)},
$$
with $p^*=2p/(p-1)$. Since $\uu-\vv\in W^{1,2}_0(B_r(x_0);\R^m)$, we have by the Sobolev's inequality $$
\|\uu-\vv\|_{L^{p^*}(B_r(x_0);\R^m)}\le Cr^\gamma\|\D(\uu-\vv)\|_{L^2(B_r(x_0);\R^m)},\quad \gamma=1-n/p,
$$
and thus \begin{align*}
    \int_{B_r(x_0)}2(\uu-\vv)\mathbf{b}\D\uu&\le Cr^\gamma\|\D\uu\|_{L^2(B_r(x_0);\R^m)}\|\D(\uu-\vv)\|_{L^2(B_r(x_0);\R^m)}\\
    &\le Cr^\gamma\left(\|\D\uu\|_{L^2(B_r(x_0);\R^m)}^2+\|\D(\uu-\vv)\|_{L^2(B_r(x_0);\R^m)}^2\right)\\
    &\le Cr^\gamma\int_{B_r(x_0)}\left(|\D\uu|^2+|\D\vv|^2\right).
\end{align*}
Combining this and \eqref{eq:alm-min-ex} yields $$
\int_{B_r(x_0)}\left((1-Cr^\gamma)|\D\uu|^2+2|\uu|\right)\le \int_{B_r(x_0)}\left((1+Cr^\gamma)|\D\vv|^2+2|\vv|\right).
$$
This implies $$
(1-Cr^\gamma)\int_{B_r(x_0)}\left(|\D\uu|^2+2|\uu|\right)\le(1+Cr^\gamma)\int_{B_r(x_0)}\left(|\D\vv|^2+2|\vv|\right)
$$
and hence we conclude that for $0<r<r_0(n,p,\|\mathbf{b}\|_{L^p(B_1;\R^m)})$
$$
\int_{B_r(x_0)}\left(|\D\uu|^2+2|\uu|\right)\le (1+Cr^\gamma)\int_{B_r(x_0)}\left(|\D\vv|^2+2|\vv|\right),
$$
with $C=C(n,p,\|\mathbf{b}\|_{L^p(B_1;\R^m)})$.
\end{proof}

%%%%%%%%%%%%%%%%%%%%%%%%%%%%%%%%%%%%%%%%%%%%%%

\end{document}